\documentclass[12pt,a4paper,twoside]{article}
\usepackage[T1]{fontenc}

\usepackage{enumerate}
\usepackage{amsthm,amsmath,amssymb,graphics}
\usepackage{graphicx}
\usepackage{psfrag}

\input xy
\xyoption{all}

\newcommand{\field}[1]{\mathbb{#1}}
\newcommand{\R}{\field{R}}

\newcommand{\N}{\field{N}}

\newcommand{\Z}{\field{Z}}

\newtheorem{theorem}{Theorem}

\newtheorem{proposition}{Proposition}[section]

\newtheorem{corollary}[proposition]{Corollary}
\newtheorem{definition}[proposition]{Definition}

\newtheorem{lemma}[proposition]{Lemma}

\newtheorem{remark}[proposition]{Remark}

\title{\textbf{ SIMPLICIAL APPROXIMATION AND COMPLEXITY GROWTH } }
\author{\small DANIEL J. PONS}
\date{\small \textit{ \hspace{13cm}To Antonia} }

\begin{document}

\maketitle


\begin{footnotesize}
\textbf{Abstract:} This work is motivated by two problems: 1) The approach
of manifolds and spaces by triangulations. 2) The complexity growth in sequences
of polyhedra. Considering both problems as related, new criteria and methods for
approximating  smooth manifolds are deduced. When the sequences of polyhedra are obtained by
the  action of a discrete group or semigroup, further control is 
given by geometric, topologic and complexity observables. We give a set of relevant examples
to illustrate the results, both in infinite and finite dimensions. 
\end{footnotesize}



\section{Introduction}

Analysis situs, an ancestor of modern topology, arose as a clandestine area of mathematics
in the nineteenth century. Gradually it became more accepted, thanks to the work of 
H. Poincar\'e, P. Alexandrov, O. Veblen, H. Hopf, J. Alexander, A. Kolmogorov, H. Weyl, L. Brouwer, H. Whitney,
W. Hodge and S. Lefschetz, among others.

One of its driving forces, the approximation of shapes (and spaces) through the juxtaposition of 
prisms or polyhedra, permeated to science and art,  
becoming   essential  in our view of the world.  From P. Picasso's cubism to quantum gravity,
human perception seemed  to  accept  simplices  as elementary blocks to approach
forms and space.

It is standard, from a mathematical perspective, to infer estimates of
error, complexity, and changes in both complexity and error in a process of approximation,
also to estabilish quantitative and qualitative
criteria for convergence, and infer bounds for the speed at which such a 
convergence (if any) occurs.

This paper is motivated by those problems; we obtain, using suitable tools,
results of this kind for evolving polyhedra on manifolds. We describe sequences
of complexes associated to coverings of spaces by open sets. Sequences
of this type, considered by P. Alexandrov (see \cite{Alexandroff}) under the name 
of projective spectra, yield, under suitable convergence
assumptions, approximations of a paracompact Hausdorff space up to homeomorphism.

We regard the number of simplices and the dimension of each complex in the sequence as a measure of 
its complexity, and control its growth not only in the limit, but also at every stage. This delivers sequences 
of irreducible complexes, those for which the excess of complexity is eliminated, say.
If the space in question is a differentiable manifold endowed with a Riemannian metric, those irreducible complexes,
together with  available tools from  geometric measure theory, yield a quantitative 
approximation as well.

To perform those constructions in a systematic way, we consider actions of discrete groups and semigroups,
say $\Gamma$, on complexes associated to coverings by open sets.
We describe $\Gamma$-representations/actions that yield convergent sequences of 
complexes, to make a connection with expansive systems, or e-systems; in those
systems the convergent sequence of complexes is obtained by iteration of a 
suitable initial simplicial complex, a generator, say.

If the space where $\Gamma$ acts expansively is a closed Riemannian manifold, estimates for 
the minimal complexity of the generating complex are achieved. This is possible
thanks to  comparison results in differential  geometry.

\bigskip
We briefly mention the contents of this work.

Sections \ref{opencovers-convergence}, \ref{dimension} 
and \ref{multi-index} provide some notation and framework.

In Section \ref{complexity-irred}  complexity functions for simplicial complexes are 
proposed, and we mention  their main and useful properties. 

Section \ref{realization} deals with concrete realizations of complexes in Euclidean space; 
this is needed, together with the functions introduced in Section \ref{complexity-irred}, to obtain
better approximations of spaces when  compared with those achieved by arbitrary 
convergent sequences (Section \ref{ponomarev}); this is developed
in Section \ref{metric} both from a qualitative and quantitative perspective. 

In Section \ref{control-convergence} increasing sequences of numbers
control  the complexity growth in sequences of complexes constructed from
finer and finer coverings, as measured by the functions introduced
in Section \ref{complexity-irred}, yielding a quantitative description of
the process in the limit. Those growths are measured by what we call 
the simplicial growth up to dimension $k$, denoted by $\text{ent}_{k}$,
and by the dimension growth, denoted by $\text{Dim}$. In fact
$\text{ent}_{k}$ is a generalization of what is known as 
topological entropy (see \cite{Wal}), meanwhile $\text{Dim}$
is a relative of mean dimension (see \cite{holomorphic-maps}).

\bigskip


Section \ref{dynamics} begins with  a natural framework for groups and semigroups
actions on spaces, usually known as $\Gamma$-spaces. We mention the natural morphisms 
between objects of this type, some advantages of this perspective, to define the
evolution of simplicial  complexes in $\Gamma$-spaces, where the growths of complexity   
can be measured.


In Section \ref{amenable} the exponential growth of the $0$-simplices  is studied under
assumptions on $\Gamma$, to infer some quantitative control  at every stage.


In Section \ref{expansivity} we describe a particular type of $\Gamma$-spaces,
namely $\Gamma$-spaces with property-e. The first remarkable issue of  
the expansive property, or property-e,  
is that it  can be characterized using  either
topological (set theoretic) or geometric tools. The set theoretic characterization leads to the concept of a
generator, an open cover that has a good response to the action of $\Gamma$, say. It could be 
seen as a  complex that under the action of $\Gamma$
evolves towards an acceptable approximation of the space. We describe  
in which sense the evolving  nerves of generating  covers 
approximate the space, and   recall fundamental 
results in geometry and topology that suit our developments. 
All the results from  previous Sections can be used in this scenario, 
and the adaptation of them is left to the reader.

In Section \ref{good-generator}, assuming that the
space is of Riemannian type, we provide estimates
to have a better control of the generating process. Those estimates
find concrete applications in Section \ref{geometrization}.

\bigskip


Finally, in Section \ref{graphs} and Section \ref{geometrization} we present
some examples. Section \ref{graphs}
deals with infinite dimensional examples where estimates for
the simplicial growth, as measured by the family 
$\{\ \text{ent}_{k}\ \vert\ k \in \mathbb{N}\ \}$, 
and the dimension growth, as measured by $\text{Dim}$, appear.
Section \ref{geometrization} describes finite dimensional closed  manifolds
for which an expansive action can be constructed. Some of the examples in
finite dimension
are not new, and the list of examples is far from being exhaustive nor definitive; their 
(not so detailed) description
is included for many purposes:
\begin{enumerate}
\item To ensure that the results of  Sections \ref{expansivity} and \ref{good-generator} are non-void, 
enabling concrete constructions and estimates. 
\item To have an idea of the methods used to construct them.
\item To allow the construction of new examples from known ones.
\end{enumerate}
Sections \ref{graphs}
and \ref{geometrization} are not entirely
independent: all the examples in Section \ref{geometrization} can be used in 
Section \ref{shifts} to construct infinite dimensional closed manifolds with property-e.

\smallskip


\section{Simplicial complexes, complexity and convergence } \label{main-properties}

We state properties of the canonical 
simplicial complex associated to the covering of a space by open sets, known as the
nerve of the covering. Some statements can be found in \cite{Alexandroff}, 
\cite{hurewicz}, \cite{Lefschetz}, and the references therein. Other 
properties are new (at least for the author), and all of them will be used 
in this  article.

\subsection{The nerve of open covers}   \label{opencovers-convergence}

If $V$ is a compact Hausdorff space\footnote{Some of the constructions and results 
are valid in more general 
spaces, but our intention is to
provide  examples in compact manifolds, usually without a boundary.} we denote by $\mathcal{C}_{V}$
the set of  covers of $V$ by open sets: one calls the members of
$\mathcal{C}_{V}$ \textbf{open covers}.

\begin{remark}
Since $V$ is compact, it suffices to identify $\mathcal{C}_{V}$ with
the totality of all \textbf{finite} covers by open sets of $V$ to simplify.
\end{remark}

If $\alpha$ and $\beta$ belong to $\mathcal{C}_{V}$, one says
that $\alpha$ is \textbf{finer} than $\beta$ if whenever $A$
is an element in $\alpha$ there exists some $B$ in $\beta$
such that $A \subseteq B$, and writes $\alpha \succ \beta$
if that is the case. This notion induces a partial order on
$\mathcal{C}_{V}$.

If $\{  \alpha, \beta  \} \subset \mathcal{C}_{V},$ one denotes
by $\alpha \cap \beta$ the refinement of $\alpha$ by $\beta$
(or equivalently the refinement of $\beta$ by $\alpha$): its elements
are intersections of one element from $\alpha$ and another
from $\beta$. One can write
$$\alpha \cap \beta := \text{l.u.b.} \{\ \gamma\ \vert\ \gamma \succ \alpha, \gamma \succ \beta\  \}  ,   $$
where $\text{l.u.b.}$ denotes the supremum (or least upper bound)
in $\mathcal{C}_{V}$ induced  by $\succ$.


\begin{remark}
One can play further with those notions and use the language of
lattices, something that we give for granted.
\end{remark}


Let $\alpha$ be given as $\{\ A_{i}\ \vert\ i \in I\   \}$, where $I$ is an
indexing  set
(finite since $V$ is compact).
Associated to $\alpha$ is a simplicial complex, known as
the \textbf{nerve} of $\alpha$, that we denote by $K(\alpha)$,
uniquely defined up to homotopy, and whose simplices
are constructed as follows:
for every $k$ in $\N$ the set of $k$-dimensional \textbf{simplices}
of $K(\alpha)$, denoted by $\triangle_{k}(\alpha)$, is given by
$$\{\ [ a_{i(0)},..., a_{i(k)} ]\ \vert\ \bigcap_{r=0}^{k} A_{i(r)} \neq \emptyset         \ \},$$
where for each $i$ in $I$ we identify the open set $A_{i}$ with
the $0$-simplex $[a_{i}]$.


\subsection{Dimension, simplicial mappings and irreducibility} \label{dimension}

Given $\alpha$ in $\mathcal{C}_{V}$, for every $k$ we denote by
$\lvert  \triangle_{k}(\alpha)     \rvert$ the cardinality of $\triangle_{k}(\alpha)$,
i.e. the number of $k$-simplices in $K(\alpha)$. By those means
one introduces the \textbf{dimension} of $K(\alpha)$, denoted by
$\text{dim} K(\alpha)$, as the maximal $k$ for which $\lvert  \triangle_{k}(\alpha)     \rvert$
is different from zero.

For $\alpha$ and $\beta$ in $\mathcal{C}_{V}$ with
$\alpha \succ \beta$ there exists  a simplicial map from
$K(\alpha)$ to $K(\beta)$, say $T_{\beta}^{\alpha} : K(\alpha) \to K(\beta)$,
defined up to homotopy, satisfying the following properties:

\begin{enumerate}
\item If $A_{i} \subseteq B_{j},$ then $T_{\beta}^{\alpha}[a_{i}] = [b_{j}] $.

\item Whenever $k > 0$ and $\sigma$ is in $\triangle_{k}(\alpha)$,
then the image of $\sigma$ under $T_{\beta}^{\alpha}$ is completely determined
by the image of the $0$-simplices making up $\sigma$: this allows
the possibility that $T_{\beta}^{\alpha} \sigma$ is in $\triangle_{l}(\beta)$
for some $l \leq k$ (for example when different vertices of $\sigma$ are
mapped to the same $0$-simplex in $K(\beta)$).
\end{enumerate}

One says that $T_{\beta}^{\alpha}$ is \textbf{compatible} with $\succ$.
It is important to note:

\begin{enumerate}
\item Such a map need not be unique.
\item If $\alpha \succ \beta \succ \gamma$ and we have constructed
two simplicial maps $T_{\beta}^{\alpha} : K(\alpha) \to K(\beta)$ and
 $T_{\gamma}^{\beta} : K(\beta) \to K(\gamma)$ compatible
 with $\succ$, then we have a simplicial map
 $T_{\gamma}^{\alpha} : K(\alpha) \to K(\gamma)$
 given by $T_{\gamma}^{\alpha} = T_{\gamma}^{\beta} \cdot T_{\beta}^{\alpha} $
 that is also compatible with $\succ$.
\end{enumerate}


There are open covers we  distinguish for later purposes.

\begin{definition}      \label{irreducible}
One says that $\alpha$ in $\mathcal{C}_{V}$  is \textbf{irreducible}
if no open refinement of $\alpha$ has a nerve isomorphic with
a proper sub-complex of $K(\alpha)$, i.e. if there is no $\beta$
finer than $\alpha$ that admits a strict simplicial embedding from
its nerve to the nerve of $\alpha$.
\end{definition}

\begin{lemma}  \label{irred-prop} Irreducible covers have the following properties:

\begin{enumerate}
\item  \label{essential} If $\alpha$ is irreducible then every member of it contains
a point in $V$ that is not contained in  other  member.

\item \label{surjective} If $\alpha$ is irreducible then whenever $\beta \succ \alpha$
all the simplicial maps from $K(\beta)$ to $K(\alpha)$ compatible
with $\succ$ are surjective.

\item \label{cofinal}If $V$ is compact, then every $\alpha$ in $\mathcal{C}_{V}$
has an irreducible refinement (one says that irreducible covers are
cofinal in $(\mathcal{C}_{V} , \succ))$.

\item \label{manifold} If $V$ is a manifold whose (real) dimension is
equal to $n$, then for every irreducible $\alpha$ in $\mathcal{C}_{V}$ one has
$\text{dim}\ K(\alpha) \leq n$.

\end{enumerate}
\end{lemma}

\begin{proof}
\ref{essential}: If $\alpha = \{\  A_{i}\ \vert\ i \in I \   \}$ has some
element, say $A_{j}$,  that  contains no point that is not contained
in the rest of the $A_{i}$'s, then $\alpha' = \{\  A_{i}\ \vert\ i \in I  \smallsetminus \{  j \}\   \}$
is a refinement of $\alpha$ and $K(\alpha')$ is a proper subcomplex
of $K(\alpha)$.

\bigskip

\ref{surjective}: This is clear from the definition.

\bigskip

 \ref{cofinal}: Let $\alpha$ in $\mathcal{C}_{V}$ be given, and consider a sequence
 $\{ \alpha_{n} \}_{ n \in \N }$ in $\mathcal{C}_{V}$ so that
 $\alpha_{0} \prec \alpha_{1} \prec \alpha_{2} \prec ...  $, where all the
 $\alpha_{n}$'s are reducible, and such that the corresponding nerves
 form a sequence of complexes, with $K(\alpha_{n+1})$ being
 a proper subcomplex of $K(\alpha_{n}),$ where $\alpha_{0} = \alpha$.
 By compacity of $V$ the sequence must stop, and the last term is
 an irreducible refinement of $\alpha$.

 \bigskip

 \ref{manifold}: Being $V$ a manifold of dimension $n$, it suffices to
 prove the result on some open set homeomorphic to $\R^{n}$, where the
 statement is obviously true.
 \end{proof}


\subsection{Chain complexes, homology } \label{multi-index}

To handle a better notation, given an open cover $\alpha = \{\ A_{i}\ \vert\ i \in I\   \}$,
we denote by $\wedge_{k}(\alpha)$ the
set of injective mappings $\vec{i} : \{ 0,...,k \} \to I^{k+1}$ such
that $\cap_{r=0}^{k} A_{i(r)} \neq \emptyset$, modulo permutations. By those means we
identify the $k$-simplex $[ a_{i(0)},..., a_{i(k)} ]$ with
$\sigma^{k}_{\vec{i}}$ whenever $\vec{i}$ is in
$\wedge_{k}(\alpha)$.  Therefore we
have a bijection between $\triangle_{k}(\alpha)$ and $\wedge_{k}(\alpha)$.

If $(G, +)$ is an Abelian group one identifies  $C_{k}(\alpha, G)$ with the (Abelian) group
of \textbf{$k$-chains} in $K(\alpha)$ with coefficients in $G$,
so that
$$C_{k}(\alpha, G) := \{\  \sum_{\vec{i} \in \wedge_{k}(\alpha)}  g_{\vec{i}}\ \sigma_{\vec{i}}^{k}\ \vert\
 g_{\vec{i}} \in G    \ \} .$$

\begin{remark}
Given a permutation of $(k+1)$ letters, say $\xi$, we are identifying $\sigma_{ \vec{i}}^{k}$
with $\sigma_{\xi \vec{i}}^{k} =  [ a_{\xi i(0)},..., a_{\xi i(k)} ]$ in $\triangle_{k}(\alpha)$,
although in $C_{k}(\alpha, G)$ we have
$\sigma_{\xi \vec{i}}^{k} = \text{sgn}(\xi)\  \sigma_{ \vec{i}}^{k}$, where $ \text{sgn}(\xi)$
denotes the sign of $\xi$.
\end{remark}


 Introduce the \textbf{boundary} operator,  denoted by
$\partial$, as the map that sends $k$-chains to $(k-1)$-chains
in a $G$-linear way. Since $C_{k}(\alpha, G)$ is generated by
the elements in $\triangle_{k}(\alpha)$, it suffices to define the action
of $\partial$ on the elements of $\triangle_{k}(\alpha)$.

Thus given $\vec{i}$ in $\wedge_{k}(\alpha)$ we set
$$\partial \sigma_{\vec{i}}^{k} := \sum_{r=0}^{k}\ (-1)^{r}\ \sigma_{\vec{i} \smallsetminus i(r)}^{k-1} ,    $$
where  $\sigma_{\vec{i} \smallsetminus i(r)}^{k-1} =  [ a_{i(0)},...,\widehat{a_{i(r)}},..., a_{i(k)} ] $
provided that $\sigma^{k}_{\vec{i}} =  [ a_{i(0)},..., a_{i(k)} ]$, where
$\widehat{a}$ means that $a$ is deleted.

One verifies that for every $c$ in  $C_{k}(\alpha, G)$ one has
$\partial^{2} c= \partial \partial c = 0$ in
$C_{k-2}(\alpha, G),$ i.e. the boundary of the boundary of
every $k$-chain is equal to zero.

\bigskip


Using the boundary operator one defines two subgroups of $C_{k}(\alpha, G)$:

\begin{enumerate}
\item The subgroup of \textbf{$k$-cycles}, denoted by $Z_{k}(\alpha , G)$,
and defined through
$$ Z_{k}(\alpha , G) :=  C_{k}(\alpha, G)\  \cap  \{\ c\ \vert\ \partial c = 0\      \}  .   $$
\item The subgroup of \textbf{$k$-boundaries}, denoted by $B_{k}(\alpha , G)$,
and defined through
$$ B_{k}(\alpha , G) :=  C_{k}(\alpha, G)\ \cap\ \partial C_{k+1}(\alpha, G) .    $$

\end{enumerate}

By those means the \textbf{$k$-th homology} group of $K(\alpha)$
with coefficients in $G$ is defined, namely
$$H_{k}(\alpha, G) := \frac{Z_{k}(\alpha , G)}{B_{k}(\alpha , G)}  .     $$

\begin{remark}
An algebraist would say that $H_{k}$ measures the \textbf{inexactness}
of the sequence
$$.... \rightarrow C_{k+1} \stackrel{\partial}{\rightarrow} C_{k}
\stackrel{\partial}{\rightarrow} C_{k-1} \rightarrow ...    $$

A geometer/topologist would say that $H_{k}$ measures the amount
of closed $k$-chains that are not filled in the space in question
(i.e. that are not boundaries) up to bordism.

\end{remark}


Let $H_{\ast}(\alpha , G) = \bigoplus_{i=0}^{\text{dim} K(\alpha)} H_{i}(\alpha , G) $  be
the graded $G$-module associated
to the homology   of $K(\alpha)$ with  coefficients in $G.$ In particular if $G$ is taken as $\R$,
one denotes by
$\mathcal{B}_{i}(\alpha) := \dim_{\R} H_{i}(\alpha , \R)$
the $i$-th Betti number of $K(\alpha)$. Regarding the structure of  the
complex  $( C_{\ast}(\alpha, \R) , \partial)$,
 one has the isomorphism $C_{i}(\alpha , \R) = Z_{i}(\alpha , \R) \bigoplus B_{i-1}(\alpha , \R)$.

 If no confussion arises we  identify $\textbf{c}_{i}(\alpha ),  \textbf{z}_{i}(\alpha )$ and
 $\textbf{b}_{i}(\alpha)$ with the real dimension of $C_{i}(\alpha , \R),  Z_{i}(\alpha , \R)$ and
 $B_{i}(\alpha , \R)$ respectively, whence in particular
 $\mathcal{B}_{i}(\alpha) = \textbf{z}_{i}(\alpha)-\textbf{b}_{i}(\alpha)$ and
 $\textbf{c}_{i}(\alpha) = \textbf{z}_{i}(\alpha) + \textbf{b}_{i-1}(\alpha)$ follow.

Using the previous nomenclature one defines
$ \chi_{t}(\alpha) := \sum_{i=0}^{\text{dim}K(\alpha)}  t^{i}   \mathcal{B}_{i}(\alpha)$,
so that  $\chi_{-1}(\alpha)$ is the \textbf{Euler-Poincar\'e characteristic}
of $K(\alpha)$.
The  equalities for $\mathcal{B}_{i}(\alpha)$ and
$\textbf{c}_{i}(\alpha)$ entail that $\chi_{-1}(\alpha)$ is equal to the sum
$\sum_{i=0}^{\text{dim}K(\alpha)}  (-1)^{i} \textbf{c}_{i}(\alpha)$.

\bigskip

 From the definitions/constructions
 one has the equality $\textbf{c}_{i}(\alpha) =  \lvert \triangle_{i}(\alpha)  \lvert  $ for every
 $i$ in $\N$, therefore:

\begin{lemma}   \label{euler-poincare}
Whenever $\alpha$ is in $\mathcal{C}_{V}$ one has the identity
$$ \chi_{-1}(\alpha) =   \sum_{i=0}^{\text{dim}K(\alpha)}  (-1)^{i}     \lvert \triangle_{i}(\alpha)  \lvert\ . $$
\end{lemma}


\subsection{Complexity functions} \label{complexity-irred}

In this Section we define  complexity functions for the  simplices
of an open cover on $V.$ We infer some properties of their minimizers and 
some estimates for them. The next observation is fundamental.

 \begin{lemma}    \label{irred-cuant}
For every $k$ in $\N$ and  $\alpha$ in $\mathcal{C}_{V}$ the minimum of
$\lvert \triangle_{k}(\beta) \lvert$
among those $\beta$'s finer than $\alpha$ is obtained for  irreducible $\beta$'s.  In particular,
if $\alpha$ is irreducible, then the minimum mentioned above is obtained for $\alpha$ itself.
The same is true for the sum   $\sum_{i=0}^{k} \lvert \triangle_{i}(\beta) \lvert$
and for the dimension $\text{dim} K(\beta)$.
\end{lemma}

\begin{proof} Follows from  \ref{cofinal} and \ref{surjective} in Lemma \ref{irred-prop},
namely that irreducible covers are cofinal
in the directed set $(\mathcal{C}_{V} , \succ )$:  hence if $\alpha$ is irreducible then for every $\beta$
finer than $\alpha$ all the simplicial maps from the nerve of $\beta$ to the nerve of $\alpha$
compatible with  $\succ$ are surjective.
\end{proof}

To quantify the complexity of $K(\alpha)$, that we measure in
terms of its dimension and its number of
simplices,  also by  similar quantities in $K(\beta)$
whenever $\beta$ is finer than $\alpha$,  we introduce
the functions $\text{Dim} K(\cdot),\ \mathcal{G}_{k}(\cdot)$ and
$S_{k}(\cdot)$  from $\mathcal{C}_{V}$ to $\mathbb{N}$ through:
$$ \text{Dim} K(\alpha) := \min_{\beta \succ \alpha} \text{dim} K(\beta)  ,      $$
$$\mathcal{G}_{k}(\alpha) := \ \sum_{i=0}^{k}  \lvert \triangle_{i}(\alpha)  \lvert\ $$
and
$$S_{k}(\alpha) := \min_{\beta \succ \alpha} \mathcal{G}_{k}(\beta) .$$

From  the definitions, Lemma \ref{irred-cuant}, and the identity in
Lemma \ref{euler-poincare} we observe:


\begin{lemma}  \label{Gk}
For every $\alpha$ in $\mathcal{C}_{V}$ we have
\begin{enumerate}

\item If $k$ is larger  than zero 
$$\mathcal{G}_{0}(\alpha)\ \leq\ \mathcal{G}_{k-1}(\alpha)\ \leq\  
\mathcal{G}_{k}(\alpha)\  ,$$ 
$$ \max_{l \in \{0,...,k\}}\ \lvert \triangle_{l}(\alpha) \rvert\ \leq\ 
\mathcal{G}_{k}(\alpha)\ \leq\ (k+1) \max_{l \in \{0,...,k\}}\ \lvert \triangle_{l}(\alpha) \rvert ,   $$
with 
$$\lvert \triangle_{l}(\alpha) \rvert\ \leq\ \binom{\mathcal{G}_{0}(\alpha)}{l+1}\ \leq\ 
\frac{\mathcal{G}_{0}(\alpha)^{l+1}}{(l+1)!} .$$

\item $\text{Dim} K(\alpha)$ is equal to $\text{dim} K(\beta)$ for some
irreducible $\beta$ finer than $\alpha$. 

\item  $S_{k}(\alpha)$ is equal to
$\mathcal{G}_{k}(\beta)$ for some irreducible $\beta$ finer than $\alpha$.

 \item  The identity
 $$ \chi_{-1}(\alpha) = 2\ \sum_{i=0}^{\text{dim}K(\alpha) -1} (-1)^{i}\
 \mathcal{G}_{i}(\alpha)
 + (-1)^{\text{dim}K(\alpha)}\  \mathcal{G}_{\text{dim} K(\alpha)}(\alpha) .   $$
\end{enumerate}
\end{lemma}


\subsection{Euclidean realization of nerves}     \label{realization}

Let $\alpha = \{\ A_{i}\ \vert\ i \in I\ \}$ be an open cover for $V$. We say that a 
partition of unity for $V$ is compatible with  $\alpha$ if it satisfies the
following conditions:

\begin{enumerate}
\item $\sum_{i \in I} x_{i}(v) = 1$ for every $v$ in $V$.
\item For every $i$ in $I$ we have that $x_{i}(v) = 0$ whenever  $v$ is not in $A_{i}$.
\end{enumerate}

Identify the $0$-simplex  $[a_{i}]$  of $K(\alpha)$ corresponding to
$A_{i}$ with the unit vector in $\R^{\lvert I \rvert}$
along the $i$-th direction, to denote the image of the map
$$ x : V \longrightarrow \R^{\lvert I \vert}   $$
$$ \hspace{3cm} v \mapsto x(v) = \sum_{i \in I}\ x_{i}(v) [a_{i}]   $$
by $\lvert K(\alpha) \rvert$, and call it an \textbf{Euclidean realization} of $K(\alpha)$.
Observe that different partitions of unit on $V$ compatible with $\alpha$ induce 
maps from $V$ to $\R^{\lvert I \rvert}$ that are homotopic. 

\bigskip

Sometimes we identify $\lvert K(\alpha) \rvert$ with a polyhedral current 
in $\R^{\lvert I \rvert}$. We do this as follows: for every $i$ in $I$ 
we have the $0$-current 
$[ a_{i} ]$ that corresponds to the pure point measure supported
at distance one from the origin along the $i$-th axis. Using the convention
of  Section \ref{multi-index}, for $\vec{i}$  in $\wedge_{k}(\alpha)$
we identify $\sigma^{k}_{\vec{i}}$ with the polyhedral $k$-current 
$\lVert \sigma^{k}_{\vec{i}} \rVert \wedge \overrightarrow{\sigma^{k}_{\vec{i}}}$,
where $\lVert \sigma^{k}_{\vec{i}} \rVert = \mathcal{H}^{k} \llcorner 1_{\text{spt} \sigma^{k}_{\vec{i}}}$
is the $k$-dimensional Hausdorff measure on $\R^{\lvert I \rvert}$ whose support 
is the convex hull of $\{ [a_{i(0)}] ,..., [a_{i(k)}] \}$, meanwhile $\overrightarrow{\sigma^{k}_{\vec{i}}}$
is a $k$-vectorfield  of unit length tangent to such a plane (see \cite{federer} for all the details).

\bigskip

Observe that the chain complex associated to
$\lvert K(\alpha) \rvert$ is isomorphic with that 
defined in Section \ref{multi-index} for $K(\alpha)$.


\subsection{Sequences of nerves: convergence} \label{ponomarev}

The results in this Section are a simplified version, suitable for the applications in this
work, of general results attributed to P. Alexandrov, S. Lefschetz and V. Ponomarev (see
\cite{Alexandroff}-\cite{Lefschetz}). In the literature the nomenclature  
is not uniform: we try to unify some notions as well.

Consider  a sequence $\{ \alpha_{n} \}_{n \in \mathbb{N}}$ in $\mathcal{C}_{V}$ with
$\alpha_{n+1}  \succ \alpha_{n}$.
If  $K(\alpha)$ is identified with  the simplicial complex that corresponds to the nerve of $\alpha$,
then  for every $n$ we have a  simplicial map $ T_{n} : K(\alpha_{n+1})   \to K(\alpha_{n}) $
compatible with $\succ$ and defined up to homotopy (see Section \ref{dimension}). Those simplicial 
maps can be composed  inductively to get 
a  map $ T^{m}_{n}$ from   $K(\alpha_{m})$  to
$K(\alpha_{n})$ whenever $m > n$
in the usual way, where $T^{m}_{n} := T_{n} \cdot T_{n+1} \cdot \cdot  \cdot\  T_{m-1}$.

We have an infinite sequence of simplicial complexes and mappings making up a
directed  set
$( K(\alpha_{n}) , T_{n}  )_{n \in \mathbb{N}}$.

One says that the sequence $( K(\alpha_{n}) , T_{n}  )_{n \in \mathbb{N}}$ is \textbf{convergent}
if every member of $\alpha_{n}$ consists  at most  of a point when $n$ goes to infinity.

As $m$ tends to infinity we have a surjective simplicial map from $K(\alpha_{m})$ to $K(\alpha_{n})$
for every $n$. We naturally identify the  inverse or projective limit of
the directed set  $( K(\alpha_{n}) , T_{n}  )_{n \in \mathbb{N}}$ with the nerve of $\alpha_{n}$ when 
$n$ tends to infinity, that we
denote by
$$    \lim_{ \leftarrow }   ( K(\alpha_{n}) , T_{n}  )  ,$$
to state:

\begin{proposition} (Alexandrov-Ponomarev \cite{Alexandroff}) \label{nerve-thread}
Assume that the sequence $$( K(\alpha_{n}) , T_{n}  )_{n \in \mathbb{N}}$$ is convergent.
Then when $n$ goes to infinity the nerve of $\alpha_{n}$ and  $V$
are homeomorphic.
\end{proposition}

\begin{proof}  We  describe  the projective limit of the directed set
$(K(\alpha_{n})  ,   T_{n})_{n \in \mathbb{N}},$ to see that there is a homeomorphism
between such a limit and $V$.

Let $\sigma := \{  \sigma_{n} \}_{n \in \mathbb{N}}$ be a sequence of simplices, with $\sigma_{n}$
in $K(\alpha_{n})$ for every $n.$ We say that $\sigma$ is an \textbf{admissible sequence} or a
\textbf{thread} for $( K(\alpha_{n}) , T_{n}  )_{n \in \mathbb{N}}$ if $\sigma_{n} = T_{n}^{m} \sigma_{m}$ whenever
$m$ is larger  than $n$, and say that an admissible sequence  $\sigma'$ is an \textit{extension} of $\sigma$
if for every $n$ the simplex $\sigma_{n}$ is a face (not necessarily a proper one) of
$\sigma_{n}'$. If the admissible sequence $\sigma$ has no extensions other than itself,
we say that it is a \textit{maximal admissible sequence} (or a \textit{maximal thread}).

Provide $( K(\alpha_{n}) , T_{n}  )_{n \in \mathbb{N}}$   with the following topology:
given a simplex $\sigma_{n}$ in $K(\alpha_{n})$ for some $n$, a basic open set around
$\sigma_{n}$ consists of all maximal admissible sequences $\sigma'$ such that
$\sigma_{n}'$ is a face of $\sigma_{n}$. In such a way one generates a topology
for the limit space, namely  the set of all maximal admissible
sequences.

Whenever $v$ is a point in  $V$ we have a simplex $\sigma_{n}(v)$ in $K(\alpha_{n})$ that
corresponds to all the open sets in $\alpha_{n}$ to which $v$
belongs; due to the  convergence assumption we note that   
$\sigma(v) = \{  \sigma_{n}(v) \}_{n  \in  \mathbb{N}}$ is a
maximal admissible sequence, and conversely, every maximal admissible sequence
in  $( K(\alpha_{n}) , T_{n}  )_{n \in \mathbb{N}}$ is of the form $\sigma(v)$
for some $v$ in $V$.

Therefore $V$ is isomorphic to the inverse limit of 
$( K(\alpha_{n}) , T_{n}  )_{n \in \mathbb{N}}$, and at this stage it is easy to see that they are homeomorphic.           
\end{proof}

\begin{remark}
Neither a metric nor a differentiable structure on $V$ are required in Proposition \ref{nerve-thread}.
\end{remark}


\begin{remark}
Proposition \ref{nerve-thread} can be refined sometimes: it might happen
that for some finite $n$  all the elements in $\alpha_{n}$ together with their
intersections are contractible (see Figure 1 in Section \ref{geometrization}). Then $\alpha_{n}$ 
is said to be a
\textit{`good cover'}, and it is  known that in such a case $K(\alpha_{n})$ is homotopically 
equivalent to $V$ (see \cite{Hatcher} for example).  
\end{remark}

One is led to consider convergent sequences of coverings to reconstruct and/or approximate 
a given space up to homeomorphism in the limit. On every paracompact 
Hausdorff space a convergent sequence can be 
constructed in an arbitrary way. It is of interest, however, to create them under
some quantitative and qualitative  control. We will see in Section \ref{control-convergence}
that the family of complexity functions introduced in Section \ref{complexity-irred}
are of much use for those purposes. Moreover, if we endow $V$ with a Riemannian metric, one can consider 
subsequences of complexes associated to those complexity functions, and  have a better approximation
of $V$ in the limit (Section \ref{metric}).


\subsection{Controlling  sequences} \label{control-convergence}

Let $( K(\alpha_{n}) , T_{n}  )_{n \in \mathbb{N}}$ be a  sequence   of nerves 
and simplicial mappings built up  from a sequence 
$\{ \alpha_{n} \}_{n \in \N}$ of open covers for $V$, with  $\alpha_{n+1}  \succ \alpha_{n}$.
From Lemma \ref{Gk} we know
that if we consider the sequence $\{ S_{k}(\alpha_{n}) \}_{n \in \mathbb{N}}$
of positive integers
there exists, for every $n$ in $\mathbb{N}$, at least one irreducible
$\beta_{k,n}$ finer than $\alpha_{n}$ so that $S_{k}(\alpha_{n})$ is
equal to $\mathcal{G}_{k}(\beta_{k,n})$.

Fix $k$ and let 
$\{ \beta_{k,n} \}_{ n \in \mathbb{N} }$ be a sequence
of irreducible covers that achieve, for each $n$ in $\mathbb{N}$, the minimum
of $S_{k}(\alpha_{n})$. Since $\beta_{k,n}$ is finer than $\alpha_{n}$,
then when $n$ goes to infinity we have, under the hypothesis of Proposition \ref{nerve-thread}, that
$K(\beta_{k,n})$ is homeomorphic to $V$; since $\beta_{k,n}$
is irreducible the dimension of $K(\beta_{k,n})$ is equal to the dimension
of $V$.

For every $i \in \{ 0,...,\text{dim} V \}$ consider the increasing sequence of positive integers
$\{ \mathcal{G}_{i}(\beta_{k,n})\}_{ n \in \mathbb{N}}$; each sequence goes to
infinity as $n$ increases. The next Proposition provides a correlation between those
sequences thanks to Lemma \ref{Gk}.

\begin{proposition}   \label{euler-convergent}
Let $V$ be a compact Hausdorff space, without a boundary, whose topological dimension
is uniform and finite.
For a fixed $k$ let $\{ \beta_{k,n} \}_{n \in \mathbb{N}}$ be a
sequence of irreducible open covers associated to a convergent
sequence $( K(\alpha_{n}), T_{n} )_{n \in \mathbb{N}}$. 
Then as 
$n$ goes to infinity we have the equality
$$ \chi_{-1}(V) = 2\ \sum_{i=0}^{\text{dim} V -1} (-1)^{i}\
 \mathcal{G}_{i}(\beta_{k,n})
+ (-1)^{\text{dim} V}\  \mathcal{G}_{\text{dim} V}(\beta_{k,n}) .   $$
\end{proposition}

To have more control on a  sequence 
$\{ K(\alpha_{n}) \}_{n \in \mathbb{N}}$ 
we  consider strictly increasing sequences of positive real numbers, say $\{ c(n) \}_{n \in \mathbb{N}}$,
going to infinity  and such that
$$ 0 < \lim \inf_{n} \frac{\log S_{k}(\alpha_{n})}{c(n)} =: \text{ent}_{k}^{\downarrow}(\alpha_{n}, c(n))
 \leq  \text{ent}_{k}^{\uparrow}(\alpha_{n} , c(n)) := \lim \sup_{n} \frac{ \log S_{k}(\alpha_{n})}{c(n)}
<   \infty.$$


If the sequence  $\{ c(n)  \}_{n \in \mathbb{N}}$ satisfies those estimates, we say
that it \textbf{controls the simplicial growth} 
of  $\{ K(\alpha_{n}) \}_{n \in \mathbb{N}}$
up to dimension $k$. If  $\lim_{n} \log S_{k}(\alpha_{n}) / c(n)$ exists, then
$$ \text{ent}_{k}^{\downarrow}( \alpha_{n}, c(n)) =   
\text{ent}_{k}^{\uparrow}( \alpha_{n}, c(n)) =: 
\text{ent}_{k}( \alpha_{n}, c(n)). $$


Similarly, if 
$$ 0 < \lim \inf_{n} \frac{\text{Dim} K(\alpha_{n})}{c(n)} =: \text{Dim}^{\downarrow}(\alpha_{n}, c(n))
\leq \text{Dim}^{\uparrow}(\alpha_{n}, c(n)) := \lim \sup_{n} \frac{\text{Dim} K(\alpha_{n})}{c(n)} < \infty ,$$
we say that $\{ c(n)  \}_{n \in \mathbb{N}}$ \textbf{controls the dimension growth}
of $\{ K(\alpha_{n}) \}_{n \in \mathbb{N}}$. 

\bigskip

Of course if  $\lim_{n} \text{Dim} K(\alpha_{n}) / c(n)$ exists, then  
$$\text{Dim}^{\downarrow}(\alpha_{n}, c(n)) = \text{Dim}^{\uparrow}(\alpha_{n}, c(n)) =:
\text{Dim}(\alpha_{n},  c(n)).$$


Using Lemma \ref{Gk} we  deduce:

\begin{theorem}   \label{controlling-growth}
Assume that  $\{ c(n)  \}_{n \in \mathbb{N}}$  controls the  simplicial growth 
of $\{ K(\alpha_{n}) \}_{n \in \mathbb{N}}$ up to dimension $k$ for some finite $k$.
Then $\{ c(n)  \}_{n \in \mathbb{N}}$ is a controlling sequence for 
the growth of simplices of $\{ K(\alpha_{n}) \}_{n \in \mathbb{N}}$ 
up to dimension $k$ for every finite $k$.
\end{theorem}

\begin{proof}
From Lemma \ref{Gk} we see that
$$\mathcal{G}_{0}(\alpha_{n}) \leq \mathcal{G}_{k}(\alpha_{n}) \leq 
(k+1)\ \max_{l \in \{ 0,...,k \}}\ \frac{\mathcal{G}_{0}(\alpha_{n})^{l+1}}{(l+1)!} , $$
therefore  we have      
$$ \text{ent}_{0}^{\downarrow}(\alpha_{n}, c(n)) \leq \text{ent}_{k}^{\downarrow}(\alpha_{n}, c(n)) \leq 
(k+1) \text{ent}_{0}^{\downarrow}(\alpha_{n}, c(n)), $$
and similarly
$$ \text{ent}_{0}^{\uparrow}(\alpha_{n}, c(n)) \leq \text{ent}_{k}^{\uparrow}(\alpha_{n}, c(n)) \leq
(k+1) \text{ent}_{0}^{\uparrow}(\alpha_{n}, c(n)) .$$

By a simple interpolation we deduce that $\text{ent}_{l}^{\downarrow}(\alpha_{n}, c(n))$ and 
$\text{ent}_{k}^{\downarrow}(\alpha_{n}, c(n))$
are comparable if both $k$ and $l$ are finite, and similarly for
$\text{ent}_{l}^{\uparrow}(\alpha_{n}, c(n))$ and $\text{ent}_{k}^{\uparrow}(\alpha_{n}, c(n))$.
\end{proof}

Natural choices for controlling sequences are:
\begin{enumerate}
\item $c(n) = n$, and then (in the strict sense)
\begin{enumerate}
\item The simplicial growth is of exponential type.
\item The dimension growth is of  linear type.
\end{enumerate}
\item $c(n) = \log n$, and  then
\begin{enumerate}
\item The simplicial growth is of  polynomial type.
\item The dimension growth is of  logarithmic type.
\end{enumerate}
\end{enumerate}

Of course  there are other possibilities.

\bigskip

Theorem \ref{controlling-growth} says that the order of the simplicial growth up to dimension $k$ 
in a sequence $( K(\alpha_{n}) , T_{n}  )_{n \in \mathbb{N}}$ is comparable to the order of growth
of  $0$-simplices if $k$ is finite or $V$ is finite dimensional.

On the other hand, 
observe that a necessary condition for $\{ K(\alpha_{n}) \}_{n \in \mathbb{N}}$ to have a  
sequence controlling  its dimension growth
is that the underlying space $V$ must have infinite  topological dimension, i.e. 
the supremum of $\text{Dim} K(\alpha)$ as $\alpha$ varies in 
$\mathcal{C}_{V}$ must be unbounded; this condition is also sufficient 
if $\{ K(\alpha_{n}) \}_{n \in \mathbb{N}}$ is convergent.

\begin{remark}   \label{standard}
We understand that $\{ n \}_{n \in \mathbb{N}}$ is the standard sequence; due to that we will omit $c(n)$
from the expressions whenever such a sequence is used.
\end{remark}


\subsection{Life with a Riemannian metric}  \label{metric}

Now $V$ is a smooth closed manifold,  provided
with a distance function arising from some Riemannian metric $g$, say $d \equiv d_{g}$. 
Then the convergence of $( K(\alpha_{n}) , T_{n}  )_{n \in \mathbb{N}}$ is equivalent to the statement
that all the members of $\alpha_{n}$ have a diameter that goes to zero as $n$
goes to infinity, where we assume that $\alpha_{n+1} \succ \alpha_{n}$
for every $n$.

As in Section \ref{control-convergence} consider, for every $k$ and $n$,
an irreducible cover $\beta_{k,n}$ finer than $\alpha_{n}$ so that $S_{k}(\alpha_{n})$ is
equal to $\mathcal{G}_{k}(\beta_{k,n})$. 
Then for each $n$ we have a simplicial embedding from the nerve of $\beta_{k,n}$
to the nerve of $\alpha_{n}$, but there is \textbf{no} guarantee that 
$\beta_{k,n+1} \succ \beta_{k,n}$, nor that the members of $\beta_{k,n}$
are contractible.

But we can do better; since the diameter of the members of  $\alpha_{n}$ are decreasing 
as $n$ increases, then for every $n$ there exists some $m$ large enough such that every member 
of the irreducible cover
$\beta_{k,n+m}$ has a diameter smaller than some Lebesgue number
of $\beta_{k,n}$, yielding a surjective simplicial map
$ T(k)^{n+m}_{n} : K(\beta_{k,n+m}) \to K(\beta_{k,n}) $.

Assume that $k$ is fixed: we have a subsequence 
$\{ \beta_{k , \phi(n)} \}_{n \in \N} \equiv \{ \beta_{\phi(n)} \}_{n \in \N}$
of  $\{ \beta_{k, n}\}_{n \in \N} \equiv   \{ \beta_{n}\}_{n \in \N} $  made up of 
irreducible covers and endowed with surjective
simplicial maps 
$T_{\phi(n)} : K(\beta_{\phi(n+1)}) \to K(\beta_{\phi(n)}) $
making up a directed set $\{ K(\beta_{\phi(n)}) , T_{\phi(n)} \}_{n \in \N}$.
We recover Proposition \ref{nerve-thread} for the projective limit
$$   \lim_{ \leftarrow }   ( K(\beta_{\phi(n)}) , T_{\phi(n)}  )  ,  $$
although with better quantitative control. This is 
due to  the results in  Section \ref{control-convergence},
and because the dimension
of $K(\beta_{\phi(n)})$ is bounded by the dimension of $V$ 
at every stage (Lemma \ref{irred-prop}).

If $\epsilon(\phi(n))$ is the largest diameter
of a member in 
$\beta_{\phi(n)} = \{\ B_{i}\ \vert\ i \in I(\phi(n))\ \}$,
we assume that $n$ is large enough so that $\epsilon(\phi(n))$ is smaller
than the injectivity radius of $(V, g)$.
Then for every $i$ in  $I(\phi(n))$ we can choose some $b_{i}$
in $B_{i}$ such that $d(b_{i}, b_{j}) \leq \epsilon(\phi(n))$ whenever $B_{i} \cap B_{j} \neq \emptyset$,
and identify $b_{i}$ with the $0$-simplex $[b_{i}]$ that corresponds to
$B_{i}$.

Let $x_{\phi(n)} : V \to \lvert K(\beta_{\phi(n)}) \rvert$ be an Euclidean realization
of $K(\beta_{\phi(n)})$ (Section \ref{realization}). Embed 
the $1$-simplices of $\lvert K(\beta_{\phi(n)}) \rvert$ in $V$ using a
Lipschitz map $y_{\phi(n)}^{1}$ so that 
the  image of the
$1$-simplex $[b_{i}, b_{j}]$ corresponds to the distance minimizing path
or geodesic between the points $b_{i}$ and $b_{j}$, that we regard
as a rectifiable path (or current) in $V$. We observe (see \cite{verde}):

\begin{lemma}
Endow the set of  
$0$-simplices in $\lvert K(\beta_{\phi(n)}) \rvert$, 
namely $\triangle_{0}(\lvert K(\beta_{\phi(n)}) \rvert)$,
with the 
distance  induced by the embedding $y_{\phi(n)}^{1}$ of the $1$-simplices in $(V, d)$,
extending it  to all $\triangle_{0}(\lvert K(\beta_{\phi(n)}) \rvert)$ in the 
natural way; denote such a distance by $d^{0}_{\phi(n)}$. Then, as $n$ goes to
infinity, the metric space 
$( \triangle_{0}(\lvert K(\beta_{\phi(n)}) \rvert) , d^{0}_{\phi(n)})$  converges to $(V,d)$  
in the Gromov-Hausdorff sense.
\end{lemma}

Consider   $\mathcal{D}_{\ast}(V) $,  
the graded $\Z$-module of  
general currents on $V$, and
for every  $\alpha$ in $\mathcal{C}_{V}$ let  $\mathcal{P}_{\ast}(\lvert K(\alpha) \rvert)$
denote the graded $\Z$-module of \textbf{polyhedral} currents on the Euclidean realization
of  $K(\alpha)$.  Hence if $y : \lvert K(\alpha) \rvert \to V$ is of
Lipschitz type, then we get a linear map
$y_{\sharp} : \mathcal{P}_{\ast}(\lvert K(\alpha) \rvert) \to \mathcal{D}_{\ast}(V)$.\footnote{See \cite{federer} for more details.}

Let $\mathbf{M} \equiv \mathbf{M}_{g}$ be the mass norm on  $\mathcal{D}_{\ast}(V)$ 
induced by the Riemannian metric $g$. A fundamental fact, 
to be found in \cite{federer}, asserts
that the closure in $\mathcal{D}_{\ast}(V)$ with respect to $\mathbf{M}$ of pushforwarded 
polyhedral currents by 
Lipschitz maps into $V$ is $\mathcal{R}_{\ast}(V)$, the
$\Z$-module of \textbf{rectifiable} currents on $V$. An important
sub-module of $\mathcal{R}_{\ast}(V)$, denoted by $\mathcal{I}_{\ast}(V)$,
is the $\Z$-module of \textbf{integral} currents; it 
consists of rectifiable currents whose boundary is also rectifiable. 
The choice of an atlas on $V$ and elementary constructions from geometric measure
theory give: 

\begin{proposition}
$\mathcal{R}_{\ast}(V)$ and $\mathcal{I}_{\ast}(V)$  
depend on the  Lipschitz structure on $V$ chosen.
\end{proposition}

\bigskip

Embed now the $2$-simplices  of
$\lvert K(\beta_{\phi(n)}) \rvert$  in $V$ by means of a Lipschitz map $y_{\phi(n)}^{2}$
using the geodesics that correspond to the
$1$-simplices as their boundary, straightening them as much as
possible, so that in the process their mass (with respect to $g$) tends to
minimize.

Then proceed inductively to  get, for every $n$ and each $j$ not bigger than the dimension of $V$,  
a Lipschitz embedding 
$y_{\phi(n)}^{j} : \triangle_{j}(\lvert K(\beta_{\phi(n)}) \rvert) \hookrightarrow V$ 
of the $j$-simplices in the Euclidean realization of $K(\beta_{\phi(v)})$ inside $V$,  all whose
images have a diameter not larger than $\epsilon(\phi(n))$;
we obtain a map at the level of  currents
$$y_{\phi(n) \sharp} : \mathcal{P}_{\ast}(\lvert K(\beta_{\phi(n)}) \rvert) \hookrightarrow \mathcal{I}_{\ast}(V)
\subset \mathcal{R}_{\ast}(V) .$$
More precisely\footnote{Visit Section \ref{realization}, if needed, for the nomenclature used.}, if $\vec{i}$ is in 
$\wedge_{j}(\beta_{\phi(n)})$, then  
$y_{\phi(n) \sharp} \sigma^{j}_{\vec{i}}$
is  almost minimal among those rectifiable currents whose boundary is 
$ \partial  y_{\phi(n) \sharp} \sigma^{j}_{\vec{i}} = y_{\phi(n) \sharp}\partial \sigma^{j}_{\vec{i}}$,
for every $j \leq \text{dim} V$. By almost minimal we mean that the minimum of mass might not occur,
however there is a sequence of Lipschitz maps leading to an infimum.

This process of approximation gives:


\begin{theorem}     \label{nerve-convergence}
Let $(V,g)$ be a smooth closed Riemannian manifold. 
Let $\{ \alpha_{n} \}_{n \in \N}$
be a sequence in $\mathcal{C}_{V}$, with $\alpha_{n+1} \succ \alpha_{n}$, and such that the diameter of each
member of $\alpha_{n}$ goes to zero as $n$ goes to infinity. Then for every $k$ and every positive $\epsilon$  
there exists some $m \equiv m(\epsilon)$, a cover $\beta_{k,m}$ minimizing 
$ S_{k}(\alpha_{m})$, and a Lipschitz map  
$y_{m} : \lvert K(\beta_{k,m}) \rvert \to V $
such that
$$ \mathbf{M}_{g}(\ y_{m \sharp} \lvert K(\beta_{k,m}) \rvert - \mathcal{V}\ ) < \epsilon  ,$$ 
where $\mathcal{V} = \lVert \mathcal{V} \rVert \wedge \vec{\mathcal{V}}$ is the  current representing $V$,
and $\lvert  K(\beta_{k,m}) \rvert$ is the polyhedral current that corresponds to an Euclidean
realization of $K(\beta_{k,m})$. If the dimension of $V$ is not $4$, such a map is independent
of the Lipschitz structure on $V$.
\end{theorem}

\begin{remark}    \label{lipschitz-equivalent}
We know, thanks to the work of E. Moise, S. Donaldson and D. Sullivan, that only in
dimension  $4$ there are smooth manifolds that are homeomorphic but not Lipschitz 
equivalent. 
\end{remark}


\section{The category of $\Gamma$-spaces} \label{dynamics}

We denote by $\Gamma$ a countable or discrete group or semigroup whose cardinality is $\aleph_{0}$. 
Let $\rho : \Gamma \to \text{Map}(V,V)$
be a representation of $\Gamma$ on the set of mappings of $V$, where we understand,
if $\Gamma$ is a group, that $\rho(\Gamma)$ is a sub-group of $\text{Homeo}(V)$,
the group of homeomorphisms on $V$. If $\Gamma$ is a semigroup, then 
$\rho(\Gamma)$ is a sub-semigroup of $\text{End}(V)$, the semigroup
of endomorphisms on $V$.  We denote a structure of this type by a tuple $(V, \Gamma, \rho)$,
and speak of a system, a representation of $\Gamma$, or a $\Gamma$-space.

\begin{remark}
We restrict  to  discrete groups and semigroups  since:
\begin{enumerate}
\item \textbf{Some} of the results are not valid for arbitrary $\Gamma$'s.  
\item \textbf{All} the examples in this article belong to this class.
\end{enumerate}
\end{remark}

We identify the totality of those structures  with the objects in  
the category of $\Gamma$-spaces. The standard morphisms between objects in this category
are:

1- Conjugations: Two systems $(V,\Gamma, \rho)$ and $(W,\Gamma, \rho')$ are said to be \textbf{conjugated} 
if there exists a homeomorphism $x : V \to W$ that interwinds the action of $\Gamma$, i.e. a homeomorphism 
between $V$ and $W$  that is $\Gamma$-equivariant. The notion of being conjugated  is a strong equivalence 
relation; not only the  underlying (topological-geometric) spaces 
in question are homeomorphic,
moreover the dynamics induced by the maps are, up to a continuous change of coordinates say,
equivalent.

2- Factors/Extensions: $(V, \Gamma, \rho )$ is said to be an \textbf{extension} of $(W, \Gamma, \rho')$,
or $(W, \Gamma, \rho')$ is said to be a \textbf{factor} of $(V, \Gamma, \rho )$, if there is
a continuous surjection $x : V \to W$, so that whenever $\gamma$ is in
$\Gamma$ we have $x \cdot \rho(\gamma) = \rho'(\gamma) \cdot x$.

\begin{remark}
Of course if $(V, \Gamma, \rho )$ is both a factor and an extension of $(W, \Gamma, \rho')$,
then  both systems are conjugated.
\end{remark}


\begin{remark}
The advantage of using the language of categories in this context is that some operations
of algebraic topology (for example the loop functor, the suspension functor
and the smash product) can be used as self-functors.
By those means one obtains new systems from known ones (see Section \ref{higher-rank}). 
\end{remark}

Consider the action of $\rho(\Gamma)$'s inverses on elements of $\mathcal{C}_{V}$. 
If $\gamma$ is in $\Gamma$ we  have a
map $\rho(\gamma) : \mathcal{C}_{V} \to \mathcal{C}_{V}$ and also an induced map
$\rho(\gamma)^{-1} : \mathcal{C}_{V} \to \mathcal{C}_{V},$ where in the case of semigroups we
understand that $\rho(\gamma)^{-1} A$ is given by $V \cap \{\ v\  \lvert\   \rho(\gamma) v \in A\  \}$ 
for every subset $A$ of $V$. Whenever $F$
is a finite subset of $\Gamma$ and $\alpha$ is in $\mathcal{C}_{V}$ we set

$$\alpha_{F} := \bigcap_{\gamma \in F} \rho(\gamma)^{-1} \alpha. $$

In what follows we  describe  
$K(\alpha_{F})$ as $F$ increases both from  a quantitative and a 
qualitative perspective.


\subsection{Exponential simplicial growth: topological entropy}  \label{amenable}

Assume that $\Gamma$ is  generated by a finite subset of elements, say 
$H$, where we assume, if $\Gamma$ is a group, that $H$ contains all its inverses, i.e.
that $H = H^{-1}$.  Then whenever $F$ is a finite subset of $\Gamma$
we define its boundary with respect to $H$, that we denote by
$\partial_{H} F$, as the subset of $F$  made up of those $\gamma$'s such that
$h \gamma$ is not in $F$ for some $h$ in $H$.

Consider an increasing sequence
of subsets exhausting  $\Gamma$, say $\{ F(n) \}_{n \in \mathbb{N}}$.  Such a 
sequence is said to be of F\o lner type if
the quotient between $\lvert \partial_{H} F(n)  \rvert$
and $\lvert F(n) \rvert$ goes to zero as $n$ goes to infinty.
If such a sequence exists, then $\Gamma$ is said to be
\textbf{amenable} (see \cite{verde} for more about this).


\begin{remark}   \label{repetition}
Since all the results in Section  \ref{control-convergence} 
are valid in this context,  we will not repeat analogous statements unless this is relevant; 
one should replace $c(n)$ by $c(\lvert F(n) \rvert)$, 
and
$ \alpha_{n}$ by 
$$\bigcap_{\gamma \in F(n)} \rho(\gamma)^{-1} \alpha,$$
for example. We will write $\text{ent}_{k}(\alpha, \Gamma , \rho, c)$ instead of 
$\text{ent}_{k}(\alpha_{n}, c(n))$ in what follows, and similarly for
$\text{Dim}(\alpha, \Gamma , \rho, c)$.
\end{remark}


Consider the standard  sequence 
$\{ c(n)  \}_{n \in \N} = \{ n \}_{n \in \N}$ (see  Remarks \ref{standard} and  \ref{repetition}). 
The next Lemma asserts that the upper and lower limits 
in Theorem \ref{controlling-growth} for $\text{ent}_{0}$   coincide.

\begin{lemma}     \label{orstein-weiss}
Let $\{ F(n) \}_{n \in \mathbb{N}}$ be a F\o lner sequence 
for $\Gamma$.  Then the following limit exists, and is
independent of the F\o lner sequence: 

\bigskip

For every  open cover $\alpha$ 
$$\lim_{n} \frac{\log S_{0}(\alpha_{F(n)})}{\lvert F(n) \rvert} =: \text{ent}_{0}(\alpha, \Gamma , \rho) .   $$

\end{lemma}

\begin{proof}
The proof follows from a general  convergence result for
subadditive functions, known as the Orstein-Weiss Lemma; a proof can be found in \cite{holomorphic-maps}.
To use such a result it suffices to note that (see \cite{Wal}, for example) 
$$\lvert \triangle_{0}(\alpha \cap \beta) \rvert = \mathcal{G}_{0}(\alpha \cap \beta) 
\leq \mathcal{G}_{0}(\alpha)\  \mathcal{G}_{0}(\beta) \ , $$
hence
$$ S_{0}(\alpha \cap \beta) \leq S_{0}(\alpha)\ S_{0}(\beta) \ .  $$  
\end{proof}

\begin{remark}    \label{topological-entropy}
If the growth of the number of simplices in $\{ K(\alpha_{F(n)}) \}_{n \in \mathbb{N}}$ is strictly exponential, then 
$(V, \Gamma, \rho)$ is said to have non zero finite  topological entropy with respect to the cover
$\alpha$. The supremum of $\text{ent}_{0}(\alpha, \Gamma, \rho)$ among
all $\alpha$'s in $\mathcal{C}_{V}$, that might not be finite, is known as the topological entropy for
$(V,\Gamma, \rho)$, and denoted by $\text{ent}_{0}(V, \Gamma, \rho)$.
\end{remark}

Recall some  constructions/results in Section \ref{control-convergence}: for a fixed $k$ 
there exists a sequence $\{ \beta_{k,F(n)} \}_{n \in \mathbb{N}}$ of irreducible covers that
achieve, for each $n$, the minimum of $ S_{k}(\alpha_{F(n)})$. In particular, by Lemma \ref{orstein-weiss} 
we have that $\text{ent}_{0}(\alpha, \Gamma, \rho)$ is given by
$$\lim_{n} \frac{\log S_{0}(\alpha_{F(n)})}{\lvert F(n)  \rvert} ,$$ 
hence 
$$ \text{ent}_{0}(\alpha, \Gamma, \rho) = 
\lim_{n} \frac{\log \mathcal{G}_{0}(\beta_{0, F(n)})}{\lvert F(n) \rvert} .   $$

On the other hand, since
$$ S_{0}(\alpha_{F(n)}) = 
S_{0}(\ \alpha_{F(n-1)} \cap \bigcap_{\gamma \in F(n) \setminus F(n-1)} \rho(\gamma)^{-1} \alpha\ )   $$
we observe that
$$ \mathcal{G}_{0}(\beta_{0, F(n)}) = S_{0}(\alpha_{F(n)})\ \leq\ 
\mathcal{G}_{0}(\beta_{0, F(n-1)})\  S_{0}(\alpha)^{\lvert F(n) \setminus F(n-1) \rvert},$$
where we recall that $S_{0}(\alpha)$ is the best lower bound for $\mathcal{G}_{0}(\beta)$
among those $\beta$'s finer than $\alpha$, 
to infer by recursion:


\begin{proposition}  \label{stage-bound}
For every $\alpha$  the number of 0-simplices in the sequence 
$\{ K(\beta_{0, F(n)}) \}_{n \in \mathbb{N}}$
is bounded at every stage by 
$$ \mathcal{G}_{0}(\beta_{0, F(n)}) \leq  S_{0}(\alpha)^{\lvert F(n) \rvert} . $$

In particular, Theorem \ref{controlling-growth} ensures that for each $k$ we have the  bound 

$$ \text{ent}_{k}^{\uparrow}(\alpha, \Gamma, \rho) \leq (k+1) \log S_{0}(\alpha) . $$

\end{proposition}

\begin{remark}    \label{sensibility}
The sensibility of the simplicial growth with respect
to the initial condition $\alpha$ can be grasped thanks to Proposition \ref{stage-bound}.
Indeed, for every $n$ and $k$ we have 
$$ S_{k}( \alpha_{F(n)})  \leq   S_{0}(\alpha)^{(k+1) \lvert F(n) \rvert} \ .$$
\end{remark}


\subsection{Convergence of nerves: e-systems} \label{expansivity}

Let $\rho : \Gamma \to \text{Map}(V, V)$ be a representation  of the group (or semigroup)
$\Gamma$  acting on $V$, and choose some metric $d_{V}$ on $V$. We say
$(V , \Gamma, \rho)$ is \textbf{expansive}, or  has \textbf{property-e},
if there exists a constant $\epsilon$ strictly larger than zero, such that for every 
$u$ different from $v$ there exists some $\gamma$ in $\Gamma$ with
$\rho(\gamma)^{\ast}d_{V}( u , v)   := d_{V}(\rho(\gamma) u , \rho(\gamma)  v)$ larger than $\epsilon$.


\begin{remark}
Every $\epsilon$ that satisfies the condition given before is called an e-constant
for $(V , \Gamma, \rho)$. The e-constants depend on the given metric $d_{V}$, but the
existence  of those  constants does not (see Lemma \ref{expansive-diameter}), and therefore one 
can omit the metric and say that
$(\Gamma, \rho)$ acts expansively on $V,$ or that $( V , \Gamma, \rho)$ has property-e.
\end{remark}

Given $\alpha = \lbrace\  A_{i}\  \lvert\   i \in I\     \rbrace$ in $\mathcal{C}_{V}$ we say that it is a
\textbf{generator} for $(V , \Gamma, \rho),$ or a generating cover, if for every array
$\lbrace\  i(\gamma)\  \lvert\  \gamma \in \Gamma\  \rbrace$
in $I$ the intersection
$$\bigcap_{\gamma \in \Gamma}  \rho(\gamma)^{-1} A_{i(\gamma)} $$
contains at most one point. Thus if $\alpha$ is a generator for $(V , \Gamma, \rho)$ then the
maximum of the diameter of all the open sets
making up $\alpha_{F}$ decreases as $F$ increases, and it goes  to zero as $F$ exhausts $\Gamma$.

\bigskip

The next Lemma provides a  rough relation between the e-constant and the diameter
of a generator (see \cite{Wal}).


\begin{lemma}  \label{expansive-diameter}
Assume that $(V, \Gamma, \rho)$ is expansive. Let $\epsilon$ be some e-constant for some  metric
on $V$, say $d_{V}$. If $\alpha = \{\ A_{i}\ \vert\ i \in I\ \} \in \mathcal{C}_{V}$ is such that 
the diameter of each $A_{i}$ is at most $\epsilon$, then $\alpha$ is a generator for $(V, \Gamma, \rho)$.
\end{lemma}

\begin{proof}
Assume that we have a pair of points $(u,v)$ in $V$ that belong to 
$\bigcap_{\gamma \in F} \rho(\gamma)^{-1} A_{i} $ for some $i$ in $I,$ this
for every  subset $F$ of $\Gamma$. Then for every $\gamma$ in $F$
the distance between $\rho(\gamma) u$ and $\rho(\gamma) v$ is at most $\epsilon$,
the diameter of $A_{i}$; since $\epsilon$ is an e-constant then $u$ and $v$ coincide,
and $\alpha$ is a generator.
\end{proof}


\begin{remark}
In Theorem \ref{bishop-expansive} we will see that it is natural to estimate, for a fixed metric $d_{V}$
on $V$, the largest e-constant. In Corollary \ref{e-constant}, an upper bound
will be provided if $d = d_{g}$ for a Riemannian metric  $g$ that is regular enough.
\end{remark}


Our next aim is to observe that a generator should be considered as being a good initial 
condition to reconstruct the skeleton of
$V$ using $(\Gamma, \rho)$; in other words, a generator \textbf{generates} a simplicial complex
that is an acceptable approximation of the space.

For those purposes let $\{  F(n) \}_{n \in \mathbb{N}}$ be an increasing sequence  
exhausting $\Gamma$, not necessarily amenable. If $\alpha$ is a generator for $( V , \Gamma, \rho)$
then we  have an infinite sequence of simplicial complexes and mappings making up a
directed  set $( K(\alpha_{F(n)}) , T_{n}  )_{n \in \mathbb{N}}$ that is convergent in the sense
of Section \ref{ponomarev}. From Proposition \ref{nerve-thread} 
and Theorem \ref{nerve-convergence} we infer:


\begin{theorem}  \label{generator-generates}
Assume that $(V , \Gamma, \rho)$ has property-e, that $\alpha$ is a generator for $(V , \Gamma, \rho)$. 
Then
\begin{enumerate}
\item When $F$ exhausts
$\Gamma$  the  complex  $K(\alpha_{F})$ and $V$
are homeomorphic.
\item If $(V,g)$ is Riemannian, closed and smooth, then for every positive 
$\epsilon$ and every integer $k$ we can find a  subset $F \equiv F(\epsilon)$ of $\Gamma$ 
such that the minimizer $\beta_{k,F}$ of
$S_{k}(\alpha_{F})$ has the following property: there exists a Lipschitz map  sending
the polyhedral current corresponding to an  Euclidean realization of $K(\beta_{k,F})$ into 
$\mathcal{D}_{\ast}(V)$, leaving it at a distance not bigger  than $\epsilon$, with respect to
$\mathbf{M}_{g}$, from the current $\mathcal{V}$ associated to $V$. If the
dimension  of $V$ is not $4$, then the  map is independent of the
Lipschitz structure on $V$. 
\end{enumerate}
\end{theorem}


As mentioned in Section \ref{dynamics}, the notion of conjugacy is an equivalence
relation in the category of $\Gamma$-spaces. It is natural to expect  that property-e,
having both a metric and a set-theoretic characterization, will be invariant under conjugation.
The next Lemma confirms this is the case:


\begin{lemma}   \label{conjugacy}
Assume that $(V, \Gamma, \rho)$ has property-e, and let $x : V \to W$ be a 
homeomorphism between $V$ and $W.$ Then $(W, \Gamma, \rho')$ is
also expansive, where $\rho'$, the induced representation of $\Gamma$ in $W$,
is given by $\rho'(\gamma) = x \cdot \rho(\gamma) \cdot x^{-1} $
for every $\gamma$ in $\Gamma$.
\end{lemma}


\begin{proof}
We use the set theoretic characterization of property-e. Assume that 
$\alpha = \{\ A_{i}\ \vert i \in I\ \}$ is a 
generator for $(V, \Gamma, \rho)$, and let $\alpha' = \{\  A'_{i} = x(A_{i})\ \vert\ i \in I\   \}$
be its image in $\mathcal{C}_{W}$. If $\lbrace\  i(\gamma)\  \lvert\  \gamma \in \Gamma\  \rbrace$
is an array in $I$ indexed by $\Gamma$, then $\bigcap_{\gamma \in \Gamma} \rho(\gamma)^{-1} A_{i(\gamma)}$
consists at most of one point, hence so does its image under $x$. 
Since $x$ is a homeomorphism
$$x (\ \bigcap_{\gamma \in \Gamma} \rho(\gamma)^{-1} A_{i(\gamma)}\ ) =  
\bigcap_{\gamma \in \Gamma}  x \cdot \rho(\gamma)^{-1} \cdot x^{-1} \cdot x (A_{i(\gamma)}) =
\bigcap_{\gamma \in \Gamma} \rho'(\gamma)^{-1}  A'_{i(\gamma)}  ,    $$
and we conclude that $\alpha'$ is a generator for $(W, \Gamma, \rho')$.
\end{proof}


\begin{remark}    \label{kervaire}
From Theorem \ref{generator-generates} and Lemma \ref{conjugacy} we infer that 
if $( V , \Gamma, \rho )$ and $( W , \Gamma, \rho')$
are  conjugated and if we know that one of them is expansive, then the simplicial complexes associated 
to the nerve of any of their generators
evolve to complexes that are homeomorphic. This cannot be strengthened to differentiable mappings 
in full generality (see also Remark \ref{lipschitz-equivalent}).

One counterexample is provided
by  expansive actions of $\mathbb{N}^{d}$ in $S^{d}$ (see Corollary 
\ref{Sd}). If $d \geq 7$ the work of M. Kervaire and J. Milnor yields a finite number 
of homeomorphic but not diffeomorphic $S^{d}$'s.

Other counterexamples are obtained if one glues these  
spheres, by connected sum, 
on manifolds that admit an expansive action (see Section \ref{geometrization}); to
the author's knowledge, this was first done by T. Farrel and L. Jones for Anosov diffeomorphisms
on $T^{d}$ (see \cite{farrell-jones}).

\end{remark}


\subsection{Choosing good generators, estimates for the e-constant} \label{good-generator}

In what follows we consider a Riemannian manifold, compact and without a boundary,
whose dimension is $d$, and whose Riemannian metric $g$ is at least of type
$C^{2}$. If $Rc(g)$ denotes the Ricci tensor of $g$, we denote by $\lambda$
the biggest real number  so that 
$$Rc(g) \geq \lambda (d-1) g$$ 
all over $V$. Let $d_{g}$ be the distance function induced by $g$; we denote by
$\lVert B_{R}(v) \rVert_{g}$ the mass (or volume) of the ball of radius $R$
centered at $v$ whenever $R$ is a positive real number and $v$ is
some point in $V$.

Let $\mathbb{S}^{d}(\lambda)$ be the simply connected space of dimension $d$ whose sectional
curvature is everywhere equal to $\lambda$; then $\lVert \widetilde{B_{R}} \rVert_{\lambda}$
denotes the mass of any ball of radius $R$ in $\mathbb{S}^{d}(\lambda)$. If
$D$ denotes the diameter of $(V, g)$ then by
Bishop's  comparison (see \cite{besse} or \cite{verde} for example)
whenever $R \leq D$ and $t \leq 1$
one has, for every point $v$ in $V$, the estimate 
$$\lVert B_{R}(v) \rVert_{g}\ /\  \lVert B_{t R}(v) \rVert_{g} \leq 
\lVert \widetilde{B_{R}} \rVert_{\lambda}\ /\ \lVert \widetilde{B_{tR}} \rVert_{\lambda} .$$

Let $C^{V}_{R}$ denote the minimal number of balls of radius $R$ needed to cover $(V,g)$;
it is not difficult to see that such a number of balls is not bigger than the largest value of
$\lVert B_{D}(v) \rVert_{g}\ /\  \lVert B_{R/2}(v) \rVert_{g}$ as $v$ varies in $V$.
Recalling Lemma \ref{expansive-diameter} we conclude:


\begin{theorem} \label{bishop-expansive}
Assume that $(V, \Gamma, \rho)$ has property-e, and let $\epsilon = \epsilon(g)$ be an 
expansivity constant for the distance function
$d_{g}$  induced by $g$. If $Rc(g) \geq \lambda (d-1) g,$ where $d$ is the dimension
of $V$, and the diameter of $(V, g)$
is $D$, then  there exists a generator for 
$(V, \Gamma, \rho)$ whose cardinality is at most
$ \lVert \widetilde{B_{D}} \rVert_{\lambda}\ /\ \lVert \widetilde{B_{\epsilon/4}} \rVert_{\lambda} $.
\end{theorem}


Abbreviate $\lVert \widetilde{B_{D}} \rVert_{\lambda}\ /\ \lVert \widetilde{B_{\epsilon/4}} \rVert_{\lambda}$
by $\Theta(\lambda, D, \epsilon)(g) \equiv \Theta(\lambda(g), D(g), \epsilon(g)),$ and observe that
$\Theta(\lambda, D, \epsilon)(g)$ is invariant under scalings of $g$. In particular, if $\lambda$
is larger than zero, then $\Theta(\lambda, D, \epsilon)$ is not bigger than
$\Theta(\lambda, \sqrt{\frac{\pi^{2}}{(d-1) \lambda}} , \epsilon)$
by Bonnet-Myers' comparison  (see \cite{besse} or \cite{verde} again).

It becomes clear that, for a fixed $g$, better estimates for the expansivity
constant $\epsilon(g)$ will improve the  upper estimate for the minimal number
of zero simplices  that the nerve associated to a generator for  $(V, \Gamma, \rho)$ 
can have; we denote that minimal number by $\lvert \triangle_{0}(V, \Gamma, \rho)  \rvert$. 
The next definition provides an upper bound for $\lvert \triangle_{0}(V, \Gamma, \rho)  \rvert$.

 
\begin{definition}   \label{expansive-invariant}
The real number $\Theta(V, \Gamma, \rho)$ is defined as
the best lower bound for $\Theta(\lambda, D, \epsilon)(g)$ as $g$ 
varies within the Riemannian
metrics on $V$ of type $C^{2}$.
\end{definition}


To obtain a lower bound for $\lvert \triangle_{0}(V, \Gamma, \rho)  \rvert$,
consider the F\o lner sequence
$\{ F(n) \}_{n \in \mathbb{N}} $  exhausting $\Gamma$.
Let $\zeta \in \mathcal{C}_{V}$
have $\delta = \delta(g)$ as a Lebesgue number (for some metric $d_{g}$ on $V$); 
if $\alpha$ is a generator for $(V, \Gamma, \rho)$ then for $t$
large enough the open cover $\alpha_{F(t)}$ has a diameter not bigger than $\delta$,
hence  $\alpha_{F(t)}$ is finer than $\zeta$, therefore by the definition 
of  $S_{0}$ (see Section \ref{complexity-irred})  the estimate
$$\text{ent}_{0}(\zeta, \Gamma, \rho) \leq \text{ent}_{0}(\alpha_{F(t)}, \Gamma, \rho) = 
\text{ent}_{0}(\alpha, \Gamma, \rho) $$
follows, hence:

\begin{lemma}  \label{sup-generator} 
If $(V, \Gamma, \rho)$ has property-e, then the supremum of $\text{ent}_{0}(\zeta, \Gamma, \rho)$ as $\zeta$ varies in
$\mathcal{C}_{V}$ is a maximum, denoted by $\text{ent}_{0}(V, \Gamma, \rho)$, and is attained 
when $\zeta$ is a generating cover.
\end{lemma}

From Theorem \ref{bishop-expansive}, Definition \ref{expansive-invariant}, 
Proposition \ref{stage-bound} and Lemma \ref{orstein-weiss} we get a 
relation between $\Theta(V, \Gamma,  \rho),\ \lvert \triangle_{0}(V, \Gamma, \rho)  \rvert$
and $\text{ent}_{0}(V, \Gamma, \rho)$.


\begin{theorem}  \label{entropy-cardinal}
Assume that $(V, \Gamma, \rho)$ has property-e. Then 
$$ e^{\text{ent}_{0}(V, \Gamma, \rho)} \leq \lvert \triangle_{0}(V, \Gamma, \rho)  \rvert  
\leq \Theta(V, \Gamma,  \rho).$$
\end{theorem}

\begin{remark}
The sense and value of the Theorems for e-systems  depend on the taste of the reader: 
\begin{enumerate}
\item If she/he is 
interested in  constructing $V$ starting from a simplicial complex of lower
complexity, it has been proved in Theorem \ref{generator-generates} 
that a method to achieve such a task is to find a group or semigroup
$\Gamma$ together with a representation $\rho$ such that $(V, \Gamma, \rho)$
has the expansive property. 

\item Once $\Gamma$ and $\rho$ have been found, bounds on the $0$-simplices
of the initial  complex are provided by Theorem \ref{entropy-cardinal}, assuming
that $\text{ent}_{0}(V, \Gamma, \rho)$ and/or  $\Theta(V, \Gamma,  \rho)$ have been computed.

\item From  another perspective, if  upper bounds for $\Theta(V, \Gamma,  \rho)$ are available, 
she/he has upper bounds for $\text{ent}_{0}(V, \Gamma, \rho)$, the rate at which the 
approximation of $V$ is achieved, and conversely.
\end{enumerate}
\end{remark}

\bigskip

Recall that the e-constant is the minimal distance that any two points in $V$
become separated under the action of $(\Gamma, \rho)$ for  a given distance function 
(not necessarily arising from a Riemannian metric) on $V$, and it is
of interest to know its order of magnitude (how large it is).
To achieve that, observe that  if $\lambda$ and $d \geq 2$ are fixed, the function 
$\lVert \widetilde{B_{R}} \rVert_{\lambda}$
is a strictly increasing  function of $R$. 

\bigskip

Using Theorems \ref{bishop-expansive} and \ref{entropy-cardinal} we can state:

\begin{corollary}   \label{e-constant}
Let $g$ be a metric of type $C^{2}$, and let $\lambda$ be the biggest real number
satisfying $Rc(g) \geq \lambda (d-1) g$ all over $V$. Then every $\epsilon$
such that
$$\lVert \widetilde{B_{\epsilon/4}} \rVert_{\lambda}\  \leq\
\frac{\lVert \widetilde{B_{D}} \rVert_{\lambda}}{\exp( \text{ent}_{0}(V, \Gamma, \rho))}$$
is an e-constant for the system $(V, \Gamma, \rho)$ with respect to
the distance $d_{g}$.
\end{corollary}


\begin{remark}
If some e-constant $\epsilon$ for a distance function $d = d_{g}$ has been found, from 
Lemma \ref{expansive-diameter} every cover whose components have 
a diameter at most $\epsilon$ is a generator. An advantage of choosing the
biggest e-constant allowed by Corollary \ref{e-constant} is that the 
the intersection pattern of covers with larger 
diameter becomes simpler, hence the associated simplicial complex 
is of lower complexity. This gives  better upper bounds for the simplicial
growth in every dimension at every stage, as predicted by Proposition \ref{stage-bound}.
\end{remark}

\begin{remark}  \label{hitchin}
Being property-e invariant under homeomorphisms (or conjugation) of $\Gamma$-systems,
then so are the numbers associated to  $\text{ent}_{0}(V, \Gamma, \rho)$ 
and the minimal complexity that a  generator can have.
In contrast, if we intersect Remark \ref{kervaire} with  N. Hitchin's result  asserting that 
in every dimension bigger than $8$ and equal to either $1 \ (\text{mod}\ 8)$ or $2 \ (\text{mod}\ 8)$
there are exotic spheres not  admitting metrics of positive scalar 
curvature (see \cite{Law-Mic}\footnote{Thanks to Professor T. Friedrich for some key points on the subject.}), 
then the estimate in Corollary \ref{e-constant} becomes  more  intriguing (after   
normalization of  volume or  diameter, say), despite its simplicity.
\end{remark}



\section{Infinite dimensional examples}   \label{graphs}

\subsection{$\Gamma$-shifths}   \label{shifts}

Consider a compact finite dimensional  and  connected manifold, 
say $V$, and infinitely many copies of it
indexed by a discrete amenable group  $\Gamma$. Hence the total space is $\mathbb{V} := V^{\Gamma}$.
Endow $\mathbb{V}$ with the weakest topology that makes the projection
in all the copies of $V$ continuous; then by Tychonov's Lemma the space
$\mathbb{V}$ is compact for this topology. Moreover, the set of all open covers for $\mathbb{V}$,
namely $\mathcal{C}_{\mathbb{V}}$, coincides with $(\mathcal{C}_{V})^{\Gamma}$.

Elements in $\mathbb{V}$ are maps $v : \Gamma \to V$, thus the natural 
representation of $\Gamma$ on $\text{Aut}(\mathbb{V})$ is given, for each
$v$ in $\mathbb{V}$ and  every pair $\{\ \gamma, \gamma'\   \}$ in $\Gamma$, by
$$ (\rho(\gamma') v)(\gamma) = v(\gamma' \gamma) .     $$

A systematic study of $\Gamma$-spaces of this type,  of $\Gamma$-invariant
subsets of them (called $\Gamma$-subshifts), was presented in \cite{holomorphic-maps}. 
Estimates on the complexity growth of those systems in the spirit of  $\text{Dim}(\mathbb{V}, \Gamma, \rho)$ 
are provided in \cite{holomorphic-maps}, but with an 
emphasis on metric and (co)homological observables. Instead of explaining   
results from \cite{holomorphic-maps}
(an interesting task indeed !) we  obtain, for some $\vec{\alpha}$
in $\mathcal{C}_{\mathbb{V}}$, estimates
for $\text{Dim}(\vec{\alpha}, \Gamma, \rho)$ and for the family
$\{\ \text{ent}_{k}(\vec{\alpha}, \Gamma, \rho)\ \vert\ k \in \mathbb{N}\    \}$.

Consider  $\vec{\alpha} = \prod_{\Gamma} \alpha(\gamma)$ in $\mathcal{C}_{\mathbb{V}} = (\mathcal{C}_{V})^{\Gamma}$.
Then  $\vec{\alpha}(\gamma)$ denotes the component of $\vec{\alpha}$ indexed by
the coordinate $\gamma,$ namely $\alpha(\gamma)$.  Natural  constructions
attributed to K\"unneth for Cartesian products of simplicial complexes (see \cite{Hatcher} or \cite{Lefschetz},
for example) ensure that for every $k$ the set $\triangle_{k}(\vec{\alpha})$ of $k$-simplices in 
$K(\vec{\alpha})$ is decomposed in products of simplices in each of the
$K(\alpha(\gamma))$ as $\gamma$ varies, i.e.
$$ \triangle_{k}(\vec{\alpha}) = 
\coprod_{\{\vec{k}\ \lvert\ \lvert \vec{k} \rvert = k\}}\  \prod_{\gamma \in \Gamma} \triangle_{k(\gamma)}(\alpha(\gamma))  , $$
where the (disjoint) union is over all the vectors $\vec{k}$ in $\mathbb{N}^{\Gamma}$ such that
$$\lvert \vec{k} \rvert = \sum_{\gamma \in \Gamma} k(\gamma)$$ 
is equal to $k$. That decomposition extends naturally to the Abelian group
$C_{k}(\vec{\alpha}, G)$ of $k$ chains on $K(\vec{\alpha})$  with coefficients in  $(G,+)$ (see Section \ref{multi-index}),
and the boundary operator is compatible with such a decomposition (with the obvious \textit{plus}
or \textit{minus} signs).

The representation  $\rho: \Gamma \to \text{Aut}(\mathbb{V})$  induces (see Section \ref{dynamics}) an 
action on $\mathcal{C}_{\mathbb{V}}$. If $\vec{\alpha} = \prod_{\gamma \in \Gamma} \alpha(\gamma)$ 
then for every $\delta$ in $\Gamma$ we have  
$\rho(\delta) \vec{\alpha} = \prod_{\gamma \in \Gamma} \alpha(\delta \gamma)$,
i.e. the components of $\vec{\alpha}$ are translated by $\delta$, so that
$(\rho(\delta)\vec{\alpha})(\gamma) = \alpha(\delta \gamma)$.

Thus whenever $F$ is a finite subset of $\Gamma$ the refinement  of $\vec{\alpha}$
under the action of the inverse of the elements in $F$ is given by
$$\vec{\alpha}_{F} := \bigcap_{\delta \in F} \rho(\delta)^{-1} \vec{\alpha} =
\prod_{\gamma \in \Gamma} (\ \bigcap_{\delta \in F} \alpha(\delta \gamma)\ ) ,   $$
i.e. the component of $\vec{\alpha}_{F}$ in the coordinate $\gamma$ is given
by the common refinement of the subset $\{\ \alpha(\delta \gamma)\ \vert\ \delta \in F\ \}$
of  components of $\vec{\alpha}$, so that 
$\vec{\alpha}_{F}(\gamma) = \bigcap_{\delta \in F} \alpha(\delta \gamma)$.

\bigskip

Choose some $\gamma'$ in $\Gamma$ and consider $\vec{\alpha}$ as being:
\begin{itemize}
\item $\vec{\alpha}(\gamma) = \alpha_{p}$ if $\gamma = \gamma'$, where $\alpha_{p}$ 
is an irreducible cover whose nerve $K(\alpha_{p})$  has dimension  $p$, where 
$ 1 \leq p \leq d$, and $d$ is the dimension of $V$.
\item $\vec{\alpha}(\gamma) = \{ V \}$ if $\gamma \neq \gamma'$, where $\{ V \}$
is the trivial cover for $V$. 
\end{itemize}

Hence for every subset $F$ of $\Gamma$ the open cover $\vec{\alpha}_{F}$
is given by:
\begin{itemize}
\item $\vec{\alpha}_{F}(\gamma) = \alpha_{p}$ if $\gamma = \delta \gamma'$ for some $\delta$ in $F$.
\item $\vec{\alpha}_{F}(\gamma) = \{ V \}$ otherwise.
\end{itemize}

Making use of K\"unneth's relations we check that $\text{Dim} K(\vec{\alpha}_{F})$ is given by 
$$\lvert F \rvert  \text{dim} K (\alpha_{p})  = p \lvert F \rvert$$ 
for every subset $F$ of $\Gamma$, 
whence $\text{Dim}( \vec{\alpha} , \Gamma, \rho )$
is equal to $p$.

To estimate the simplicial growth we observe that
$$\mathcal{G}_{0}(\vec{\alpha}_{F}) = \mathcal{G}_{0}(\alpha_{p})^{\lvert F \rvert} ,$$ 
and a little  bit of algebra 
along the lines of Lemma \ref{Gk} ensures, making use of K\"unneth's relations, that
$$ \text{ent}_{k}(\vec{\alpha} , \Gamma , \rho) = \text{ent}_{0}(\vec{\alpha}, \Gamma, \rho) = 
\log \lvert \triangle_{0}(\alpha_{p})  \rvert $$
for every $k$.

We conclude that $\{ c(n) \}_{n \in \N} = \{ n \}_{n \in \N}$ is a 
controlling sequence both for the simplicial and the dimension growth of 
$\{ K(\vec{\alpha}_{n}) \}_{n \in \N} \equiv \{ K(\vec{\alpha}_{F(n)})  \}_{n \in \N}$.


\subsection{Pyramids and prismatic covers}

Consider on $V$  irreducible covers  all of whose members have at least one
point in common; if  $\alpha$ is  an open cover  with this property
we say that $\alpha$ is a \textbf{prismatic} cover for $V$. 
Prismatic covers satisfy:
\begin{enumerate}
\item \label{prism-1} $\text{dim} K(\alpha) = \mathcal{G}_{0}(\alpha)  -1$.
\item \label{prism-2} If $k$ is less or equal than $\text{dim} K(\alpha)$ then
$$\mathcal{G}_{k}(\alpha)  = \binom{ \mathcal{G}_{0}(\alpha)  }{k+1} + 
\mathcal{G}_{k-1}(\alpha).$$
\end{enumerate} 

Observe that if $V$ is  connected, then   a prismatic cover whose nerve has dimension 
smaller or  equal than the dimension of $V$ always exists (indeed, prismatic covers get rid of all 
the topology on $V$, if any).

Here 
$(\ \mathbb{R}^{\Gamma}, \lVert\ \rVert_{1}\ ) \equiv (\ X, \lVert\ \rVert_{X}\ )$ is the Banach space
of arrays of real numbers indexed by elements of a discrete group $\Gamma$  with
the $l_{1}$-norm. Elements in $X$ are functions $x : \Gamma \to \mathbb{R}$
such that  
$ \lVert x \rVert_{1} :=  \sum_{\gamma \in \Gamma}\ \lvert x(\gamma) \rvert < \infty $,
and a basis for this linear space is given by the maps
$\{\ e_{\gamma} : \Gamma \to \{0,1\}\ \vert\ \gamma \in \Gamma\ \}$ 
satisfying $e_{\gamma}(\gamma') = 1$ if  $\gamma = \gamma'$, and zero otherwise,
so every  element $x$ in $X$ can also be written as a sum
$\sum_{\gamma \in \Gamma}\ x(\gamma)\ e_{\gamma}$.

The predual of $(\ \mathbb{R}^{\Gamma}, \lVert\ \rVert_{1}\ )$ is the Banach space
 $(\ \mathbb{R}^{\Gamma}, \lVert\ \rVert_{\infty}\ ) = (\ Y, \lVert\ \rVert_{Y}\ )$  
of functions $y : \Gamma \to \mathbb{R}$ such that
$ \lVert y \rVert_{\infty} := \sup_{\gamma \in \Gamma}\ \lvert y(\gamma) \rvert < \infty $.
A basis for this linear space is given by the maps
$\{\ e^{\gamma} : \Gamma \to \{0,1\}\ \vert\ \gamma \in \Gamma\ \}$ 
such that $\langle e^{\gamma} , e_{\gamma'} \rangle = 1$ if $\gamma = \gamma'$,
and zero if $\gamma \neq \gamma'$, therefore $\langle e^{\gamma} , x \rangle = x(\gamma)$ for every $x$ in $X$.

\bigskip

On $X$ one can also consider families of seminorms given by 
$$p_{C}(x) := \max \{\ \lvert \langle y , x  \rangle \rvert :=
\lvert \sum_{\gamma \in \Gamma} y(\gamma) x(\gamma) \rvert\ \vert\ y \in C\ \},$$ 
where $C$ ranges
over arbitrary finite subsets on $Y$. The open sets associated to that
family of seminorms generate a Hausdorff topology on $X$;
an application of Tychonov's Lemma ensures that
bounded sets  in $X$ are compact in this topology.

\bigskip

Let $V := X \cap \{\ x\ \vert\ \lVert x \rVert_{1} \leq 1\ ,\ x(\gamma) \geq 0\ \text{for every $\gamma$}\ \}$ be 
the part of the unit ball in $(\ \mathbb{R}^{\Gamma}, \lVert\ \rVert_{1}\ )$
all of whose coordinates are non-negative, 
and consider the topology induced by the family of seminorms $\{ p_{C}  \}$. Then $V$ is  closed, convex,
Hausdorff and compact, and can be spanned by convex linear combinations of
its extremal points, say $E(V)$, that consists of the set  
$\{\ e_{\gamma}\ \vert\ \gamma \in \Gamma\ \}$
together with the origin in $X$, that we denote by $e_{0}$ if no confussion arises.
We regard $e_{0}$ as the \textbf{apex} of the \textbf{pyramid}   $V$, and say that the convex set spanned
by the rest of the extreme points is the \textbf{base} of $V$. Note that the base of
$V$ consists of maps $v : \Gamma \to [0,1]$ such that 
$\lVert v \rVert_{1} = \sum_{\gamma \in \Gamma} v(\gamma) = 1 $.

\bigskip

We pick the family  $\{ p_{C} \}$ as follows: if $F$ is a finite subset of $\Gamma$ we denote by
$p_{F}$ the seminorm  on  $X$ given by
$ p_{F}(x) := \max \{\  \lvert x(\gamma) \rvert\  \vert\ \gamma \in F\    \}  $.

The family  $\{\ p_{F}\ \vert\ F\ \text{is a finite subset of}\ \Gamma\ \}  $
of seminorms on $X$  provide the desired properties on $V$. Note that 
given $T$ bigger than zero:

\begin{itemize}
\item On one side $p_{F}(x) < T$ if and only if for every $\gamma$ in $F$ we have $\lvert x(\gamma) \rvert < T$,
hence 
$$\{\ x\ \vert\ p_{F}(x) < T\ \} = \bigcap_{\gamma \in F} \{\ x\ \vert\ p_{\gamma}(x) < T\    \} .   $$
\item On the other side  $p_{F}(x) > T$ if and only if there exists some $\gamma$ in $F$ with 
$\lvert x(\gamma) \rvert > T$, therefore
$$\{\ x\ \vert\ p_{F}(x) > T\ \} = \bigcup_{\gamma \in F} \{\ x\ \vert\ p_{\gamma}(x) > T\    \} .   $$
\end{itemize}

Choose a positive number $a$ small enough (in fact smaller than $0.5$), and for every $\gamma$ 
in $\Gamma$  define an 
open cover $\alpha[\gamma]$ with two members for $V$ as follows:

\begin{itemize}
\item The open set $A_{\gamma}$ is given by those $v$ in $V$
such that $p_{\gamma}(v) > 0.5 - a $. Thus $A_{\gamma}$ consists of those elements
in $V$ whose distance from the vertex $e_{\gamma}$ is smaller
than $0.5 + a$, the distance being the one induced by the 
norm $\lVert\ \rVert_{1}$ on $X$. 

\item The open set $A_{\gamma}'$ is given by those $v$ in $V$ such that
$p_{\gamma}(v) < 0.5 + a$, i.e. $A_{\gamma}'$ contains the apex of $V$ 
and those points whose  distance from the vertex $e_{\gamma}$
is larger than $0.5 - a$. 
\end{itemize}

We construct two examples in this setup, one that does not use the group structure of $\Gamma$
at all, meanwhile the other  uses such a structure.

\begin{enumerate}

\item \label{no-group} In this example $\Gamma$ could be any denumerable infinite set. Using the 
covers $\alpha[\gamma]$ yet mentioned we construct, for every
finite subset $F$ of $\Gamma$, an open cover $\alpha[F]$ whose cardinality is $\lvert F \rvert + 1$,
containing two types of open sets:

\begin{itemize}

\item  For each $\gamma$ in $F$ we have an open set $A_{\gamma}$, as before. Observe that
the union  $\bigcup _{\gamma \in F} A_{\gamma}$ consists of those
points in $V$ with $p_{F}(v) > 0.5 - a$.

\item An open set $A_{F}'$ is given by those $v$ in $V$ such that
$p_{F}(v) < 0.5 + a$, therefore $A_{F}'$ is the intersection
$\bigcap_{\gamma \in F} A'_{\gamma}$.
\end{itemize}

One verifies that $\alpha[F]$ is a prismatic cover for $V$ whenever $F$ is a finite subset of $\Gamma$,
also that $\alpha[F']$ is finer than $\alpha[F]$ whenever $F$ is a subset
of $F'$.

Let $\{F(n)\}_{n \in \N}$
be an increasing sequence of subsets exhausting $\Gamma$, and consider the sequence of
complexes $\{\ K(\alpha[F(n)])\ \}_{n \in \N}$. Since $\alpha[F(m)] \succ \alpha[F(n)]$
whenever $m > n$ we get a directed sequence $\{\ K(\alpha[F(n)]) , T_{n}\ \}_{n \in \N}$ of complexes
and maps, thus the constructions/results in Section \ref{control-convergence} can be used.

Being $\alpha[F(n)]$ a prismatic cover for every $n$, the simplex $K(\alpha[F(n)])$
has the maximal number of simplices allowed, say. It is easy to see that 
$\{\ \log \lvert F(n) \rvert\ \}_{n \in \N}$ controls the simplicial growth of
$\{\ K(\alpha[F(n)])\ \}_{n \in \N}$, and for every $k$ we have 
$$ \text{ent}_{k}( \alpha[F(n)] , \log \lvert F(n)  \rvert) = k +1 ,$$
showing that the estimates in Theorem \ref{controlling-growth} are sharp.

In this example  $\{ \lvert F(n) \rvert \}_{n \in \N}$ controls the dimension growth
of $\{\ K(\alpha[F(n)])\ \}_{n \in \N}$, and we see that 
$\text{Dim}(\alpha[F(n)] , \lvert F(n) \rvert )$ is  equal to one.

\item We consider a representation $\rho: \Gamma \to \text{Aut}(V)$ that leaves 
the apex fixed and translates the coordinates in the base, so that if
$$v = v(0)e_{0} + \sum_{\gamma \in \Gamma}\ v(\gamma) e_{\gamma} ,$$
then 
$$\rho(\delta) v = v(0) e_{0} + \sum_{\gamma \in \Gamma}\ v(\delta \gamma) e_{\gamma}. $$

For every $\gamma$ the open cover $\alpha[\gamma]$ is given by 
$\{ A_{\gamma} , A'_{\gamma} \}$, hence  $\rho(\delta)\alpha[\gamma]$
is just $\alpha[\delta \gamma]$. Therefore for every finite subset
$F$ of $\Gamma$ the  cover $\alpha[\gamma]_{F}$ is given, according to 
Section \ref{dynamics}, by
$$ \bigcap_{\delta \in F} \rho(\delta)^{-1} \alpha[\gamma] = 
\bigcap_{\delta \in F} \alpha[\delta^{-1} \gamma]  .  $$

The  covering  $\alpha[\gamma]_{F}$ is not irreducible if $F$ has at least
two elements, however the cover $\alpha[F^{-1} \gamma]$
constructed in \ref{no-group} is finer than $\alpha[\gamma]_{F}$ (and
prismatic). Hence for every $k$ whenever $F$ is a finite subset of $\Gamma$
we have the equality
$$ S_{k}(\alpha[\gamma]_{F}) = 
\mathcal{G}_{k}(\alpha[F^{-1} \gamma]) .   $$

As in \ref{no-group}, choose an increasing family of subsets $\{ F(n) \}_{n \in \N}$
exhausting $\Gamma$, to infer that the simplicial growth up to dimension $k$
for the sequence 
$\{\ K(\alpha[\gamma]_{F(n)})\   \}_{n \in \N}$ is  polynomial of degree $(k+1)$,
this for every $k$, and  the dimension growth is linear, and equal to one.

\begin{remark}
Being the simplicial growth of polynomial type the statements
in Remark \ref{sensibility} are not relevant.
\end{remark}

\end{enumerate}


\section{Finite dimensional manifolds and property-e}  \label{geometrization}

As explained in Section \ref{expansivity}, if $V$ admits an expansive
action of a group (semigroup)
$\Gamma$  we can, in a precise sense, reconstruct a complex that  is homeomorphic to $V$
if we take as an initial condition the nerve associated to an open cover that is a
generator for $(V , \Gamma, \rho)$. Moreover, if  $V$  is endowed
with a Riemannian metric and its dimension is bigger than one all the estimates in
Section \ref{good-generator} for the
e-constant, the (minimal) complexity of generating covers, and their relation to the (topological) entropy
provide interesting  information (sometimes without much effort).

If the dimension of $V$ is either one or two the
classification of closed orientable manifolds is complete and extremely simple.  In those cases if
$V$ and $W$ are homotopically equivalent finite and boundaryless simplicial complexes then they are 
homeomorphic, and even diffeomorphic if they are endowed with a smooth structure.

In the context of algebra, the simplest  groups and semigroups
are $\mathbb{Z}$ and $\N$ respectively. Thus to understand expansive actions of 
groups and/or semigroups on closed
orientable manifolds it is natural to begin with the simplest examples, i.e. with $\Z$ and/or $\N$ actions on closed (orientable) manifolds, to then  consider Abelian actions of products of those.

\subsection{Dimension  one } \label{S1}

The only closed one dimensional manifold up to homeomorphism
is $S^{1}.$  If $\Gamma$ is equal to $\N,$ then $(S^{1} , \N)$ is expansive if one
considers the $\N$-action $n   : \theta \mapsto   k^{n}\ \theta$ for some fixed $k$ in $\Z $
whose absolute value is bigger than one. If no confussion arises we denote such
a representation by $f$ so that $f^{n}(\theta) = k f^{n-1}(\theta)$
whenever $n$ is a natural number.

Consider for simplicity the case when $k$ is equal to two. Let $\alpha$ be the open cover
 of $S^{1}$ given by  $\{\  ] -a\ ,\ \pi + a  [\  ,\  ] \pi - a , a   [\ \}$
for some positive $a$ that is small enough. Then $\alpha$ is a generator, and it
is easy to see that 
$\text{ent}_{0}(\alpha, \mathbb{Z} , f) = \text{ent}_{0}(S^{1}, \mathbb{Z} , f) = \log 2$ (see Figure 1).


\begin{figure}[h]
\begin{center}
\includegraphics[scale=0.6]{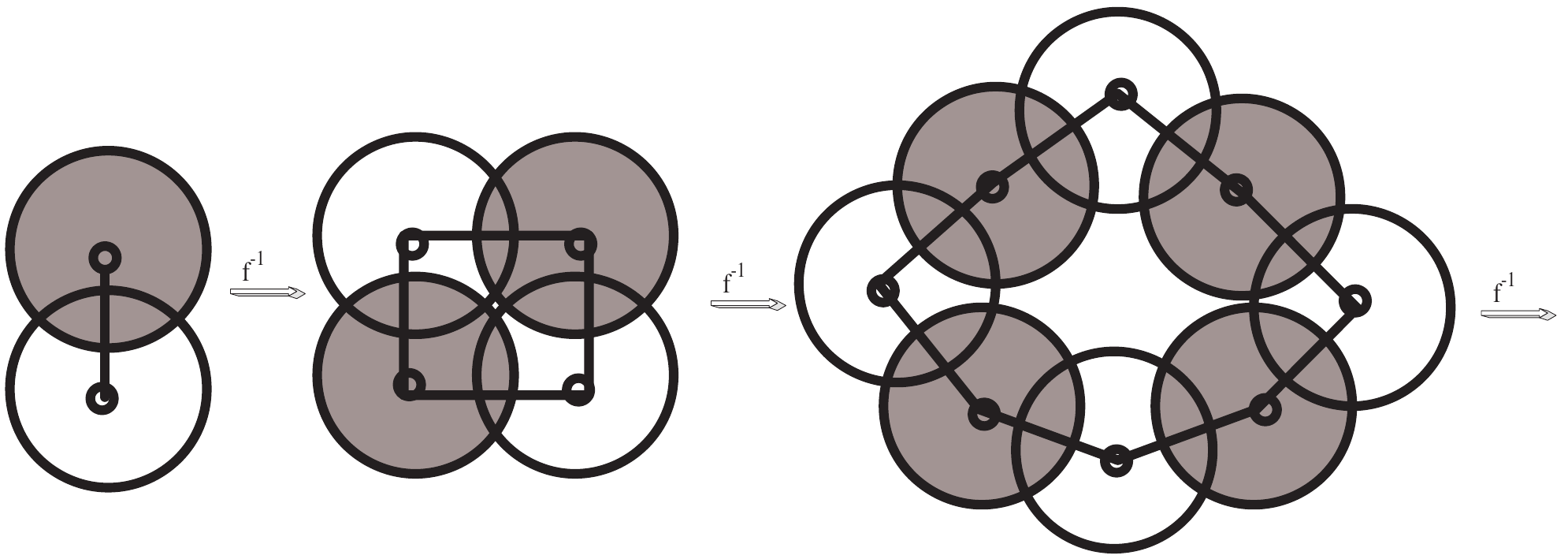}
\caption{{\scriptsize A schematic evolution of the nerve of $\bigcap_{n=0}^{T} f^{-n}\alpha$  when $f : S^{1} \to S^{1}$ is given
by $f(\theta) = 2 \theta$. Here  $\alpha$ consists of two semicircles overlapping in a neighborhood of
$\theta = 0$
and $\theta = \pi$, with $T$ being equal to $0, 1$ or $2$ (from left to right).}}
\end{center}
\end{figure}


\subsection{Dimension  two}

Closed manifolds of dimension two, also known as compact Riemann surfaces,
are the basic  test of (almost) every theory that  wishes to be extended to higher dimensions.

The classification of them up to diffeomorphism is extremely simple, and everyone can
distinguish among them by the number of holes (or the intersection form in the first homolgy
group with $\Z_{2}$ coefficients). Within the orientable ones  we will construct expansive actions of either  $\Z,  \N$
or $\N^{2}$, depending on the genus.


\subsubsection{Actions of $\Z$}

Expansive homeomorphisms (or expansive actions of $\Z$) in compact Riemann surfaces of positive
genus were constructed in \cite{obrien-reddy}. We briefly explain some ideas.

Consider the standard Anosov  homeomorphism on the 2-torus (see Section \ref{anosov} for general definitions),  namely the one induced by the
matrix
$$A =
\begin{bmatrix}
1  & 1 \\
1   & 2
\end{bmatrix} $$
on  $\R^{2}.$ Let $h : T^{2} \to T^{2}$ be the induced  homeomorphism that turns out to
be expansive, and note that if $h : W  \to W $ induces an
expansive action of $\Z$ then so does
$h^{k} : W \to W$ whenever $k$ is a positive integer.

Let $\Sigma_{g}$ denote the orientable Riemann
surface of genus $g,$ and consider a \textit{branched} cover $x : \Sigma_{g} \to \Sigma_{1} = T^{2}$ to construct a homeomorphism $f : \Sigma_{g} \to \Sigma_{g}$ by lifting $h^{k}$ through $x$ for
some $k.$ If the pair $(x , f)$ can be constructed, then $f : \Sigma_{g} \to \Sigma_{g}$ provides
an expansive action of $\Z$ on $\Sigma_{g}$ (observe that this is not true in higher dimensions because the branch set could have strictly positive dimension, and the dynamics of the lifted map, namely $f,$ need not be expansive therein).

Considering  standard relations that the map $x : \Sigma_{g} \to T^{2}$ should satisfy at the level of the fundamental groups to achieve a branched cover,  lifts of iterates of $h^{3} : T^{2} \to T^{2}$ are constructed
for every $g$ bigger than one in  \cite{obrien-reddy}, providing the desired expansive systems
$(\Sigma_{g} , \Z)$ whenever $g$ is different from zero.

Some years later K. Hiraide and J. Lewowicz  (see \cite{hiraide} and \cite{lewowicz}) found a natural relation between expansive actions of $\Z$ on hyperbolic Riemann surfaces and neat constructions/results on Teichm\"uller theory due to W. Thurston (see \cite{thurston}). The result in \cite{hiraide}-\cite{lewowicz} can be rephrased using the language
developed  in \cite{thurston}  as follows:


\begin{theorem}
(Hiraide-Lewowicz)
Let $\Sigma_{g}$ be a closed and orientable  hyperbolic Riemann surface, and assume that
$f : \Sigma_{g} \to \Sigma_{g}$ induces an expansive action of $\Z$ (those actions are known
to exist due to the constructions in \cite{obrien-reddy}). Let $\mathcal{T}(\Sigma_{g})$ denote
the Teichm\"uller space of $\Sigma_{g}$. Then for some $f^{*}$ conjugated to $f$
the induced action of  $f^{*}$ on the closure
of $\mathcal{T}(\Sigma_{g})$, denoted by
$$\overline{\mathcal{T}}(f^{*}) :  \overline{\mathcal{T}}(\Sigma_{g})  \to   \overline{\mathcal{T}}(\Sigma_{g}),        $$
has exactly two  fixed points. Those points are on the boundary of $\overline{\mathcal{T}}(\Sigma_{g})$
and correspond to  projective classes of mutually transverse measured laminations on $\Sigma_{g}.$ One of those projective classes has a representative that  expands under the action of
$f^{*}$, while the other
contracts  (one says that $f$ is conjugated to a pseudo-Anosov diffeomorphism).
\end{theorem}

In \cite{hiraide} and \cite{lewowicz} it is stated  that $S^{2}$ does not
admit an expansive action of $\Z$.


\subsubsection{Actions of $\N$} \label{simplicialvolume}

It is rather easy to see that $V$ admits an expansive action of $\N$
only if  there exists  a map $f: V \to V$ whose degree it at least two. A necessary condition for the existence of such a map is
that the simplicial volume  of $V$ is equal to zero (see \cite{verde}),
and then $V$ must be either the two sphere or the two torus. On $T^{2}$ an expansive action of $\N$ can be easily constructed, although in $S^{2}$ it is
not possible to achieve that (see  Section \ref{expansive-N} for both issues).

\bigskip

Therefore to \textit{complete the program} of reconstructing every orientable closed manifold
whose dimension is two  from a simplicial complex that has a simpler structure we are led to consider higher rank actions.


\subsubsection{Actions of $\N^{2}$}

An expansive action of $\N^{2}$ on $T^{2} = S^{1} \times S^{1}$ can be achieved using expansive
actions of $\N$ on $S^{1}$ (see Section \ref{S1}) if one considers the remarks in
Section \ref{higher-rank} concerning cartesian products of expansive systems.

An expansive action of $\N^{2}$ on $S^{2}$ is constructed as a particular case
of Theorem \ref{smash} in Section \ref{higher-rank} (see Corollary \ref{Sd}).


\subsection{Higher dimensional examples}

There exist (partial) characterizations  of expansive actions on closed manifolds of the simplest
groups and semigroups, namely   $\Z$ and $\N$ respectively.


\subsubsection{Actions of $\Z$} \label{anosov}

Let $V$  be a closed manifold admitting two foliations of complementary dimension that are
transversal all over $V,$ and let $f : V \to V$ be a diffeomorphism preserving those foliations. Assume furthermore that
$f$ stricly expands the current corresponding to one of those foliations   and strictly contracts the
other one (see \cite{sullivan}).  One says that $(V , \Z, f)$ is \textit{Anosov}, and it is easy to see that Anosov systems provide examples of expansive  $\Z$-actions.

A good introduction to Anosov systems  can be found in  \cite{smale}, and modulo
examples unknown to the author in all the   systems of this type the underlying space
is, up to conjugation, an \textit{infra-nilmanifold}, i.e. up  to a finite cover and homeomorphism, 
a co-compact quotient of a connected simply connected
nilpotent Lie group, say $G$, the quotient being induced by the action of a discrete subgroup of $G$,
say $\Upsilon$, that is finitely generated, nilpotent, and has no elements of finite order
(see \cite{smale} again), generalizing linear automorphisms on tori. If 
$f : G/\Upsilon \to G/\Upsilon$ is Anosov, then the linear map induced at the level of Lie algebras
has no eigenvalues in the unit circle, and an important part of the structure of these systems can
be decoded by algebraic means (see \cite{lauret}).

One interesting feature of infra-nil-automorphisms with the Anosov property is that the 
observables involved in the estimates in Section \ref{good-generator} can be easily found.
For instance, if $\lambda$ is the largest eigenvalue of the linearization of $f$, then
$\text{ent}_{0}(V, \N, f) $ is equal to $\log \lambda$ (see \cite{sullivan}),
extending the pseudo-Anosov behavior, where $\lambda$  corresponds to the
expansion/contraction coefficient for the transverse measured laminations.


\subsubsection{Actions of $\N$} \label{expansive-N}

In \cite{expansive-expanding} is shown that on closed manifolds
a map $f : V \to V$ represents an expansive action of $\N$ if and only if such  a map is expanding in the sense of \cite{polynomial-growth}, namely if for some metric $d$ on $V$ and every point $v$ in $V$
there exists exists a neighbourhood of that point  such that $f^{\ast} d > d$ outside the diagonal therein.

The following discussion is based on $\cite{polynomial-growth}$. It is proved  that a necessary condition  for the existence of a map of this type on a closed manifold is that their universal cover is homeomorphic to $\R^{n}$. To achieve that M. Gromov notes that the lift of those maps to the universal cover are globally expanding for some metric
invariant under  deck transformations, a condition that is easy to verify.

The simplest examples
of this kind are induced by linear maps on $\R^{n}$ whose eigenvalues are greater than one and
that are  compatible with the free action of discrete  groups on $\R^{n},$ say
$\Upsilon : \R^{n} \to \R^{n},$ so that $V = \R^{n}/\Upsilon$ is compact. An invariant metric in those examples is of course the \textbf{very}-flat canonical  one (see \cite{besse}): every flat manifold of this
type admits an expanding action of $\N$, a result that  Gromov attributes to D. Epstein and M. Shub.

Assuming an upper bound on the Jacobian of the  map one sees that a necessary condition for the existence 
of an expanding map on a closed manifold, say $V,$ is that the fundamental group
must have polynomial growth. The analogous result without using the assumption that the map
is differentiable and obtained using  techniques from geometric group theory is due
to   J. Franks.

Hence  the  candidates are closed aspherical manifolds that do no admit metrics
of negative sectional curvature (see \cite{besse} or \cite{verde}, for example).
Needless to say, those are necessary conditions.

Posterior work of Shub (see \cite{polynomial-growth}) enables to assert  that
an expanding system $(V , \N)$ is conjugated to an
infra-nil-endomorphism if and only if
the fundamental group of $V$ contains a nilpotent subgroup of finite index.

Since the main  result in \cite{polynomial-growth} claims that
every finitely generated group with polynomial growth is virtually  nilpotent,
one concludes that \textit{every} expansive $\N$-action  on a closed
manifold is conjugated to an infra-nil-endomorphism.

It is worth mentioning the result of D. Epstein and M. Shub: it provides the only known
examples of closed manifolds, that are not products, with special holonomy (see \cite{besse}, 
and the foundational \cite{cal-geo}) and of dimension larger than two, 
allowing an expansive action.
Indeed, all  complex tori belong to this class, and the estimates of
Section \ref{good-generator} enriched with the Monge-Amp\`ere-Aubin-Calabi-Yau 
developments provide a play-ground.


\subsubsection{Higher rank actions} \label{higher-rank}

Let $\{ V_{\omega} \}_{\omega \in \Omega}$ be a finite collection of closed manifolds 
so that for each $\omega$ in $\Omega$
the space $V_{\omega}$ admits an expansive action of a group or semigroup  $(\Gamma_{\omega}, \rho_{\omega}).$
By means
of the set-theoretic characterization of property-e (Section \ref{expansivity}) one readily sees
that the Cartesian product of them, say
$V := \prod_{\omega \in  \Omega} V_{\omega}$, also admits an expansive action of 
$(\Gamma, \rho) := \prod_{\omega \in \Omega} ( \Gamma_{\omega}, \rho_{\omega})$.

\bigskip

Consider now the \textit{wedge sum} of the finite collection of spaces 
$\{ V_{\omega} \}_{\omega \in \Omega}$,  denoted 
by $\lor_{\omega \in  \Omega} V_{\omega}$, where in each of the $V_{\omega}$'s a base point $v_{\omega, 0}$ 
is understood. Inside the Cartesian product of  the $V_{\omega}$'s  collapse 
the wedge (sum) of the spaces to a point, to get the $\textit{smash}$ of
$\{ V_{\omega} \}_{\omega \in \Omega}$, usually written as $\wedge_{\omega \in \Omega} V_{\omega}$.

In the category of topological spaces (with base points) the smash product is  a commutative and associative 
self-functor. Hence if $(\Gamma_{\omega}, \rho_{\omega})$ is a group or semigroup acting on $V_{\omega}$ having
the base point $v_{\omega , 0}$ as  a fixed element  for every $\omega$, then  there is a natural  action of
$\{ (\Gamma_{\omega} , \rho_{\omega}) \}_{\omega \in \Omega }$ on  $\wedge_{\omega \in \Omega} V_{\omega}$, 
denoted by
$$\wedge_{\omega \in \Omega} (\Gamma_{\omega}, \rho_{\omega}) :   \wedge_{\omega \in \Omega} V_{\omega}    \to     \wedge_{\omega \in \Omega} V_{\omega}\ , $$
that is commutative and associative with respect to the different $\omega$ coordinates (in the same way
as $\prod_{\omega \in \Omega} (\Gamma_{\omega}, \rho_{\omega}):  \prod_{\omega \in \Omega} V_{\omega} \to  \prod_{\omega \in \Omega} V_{\omega} $).

\bigskip

The next result asserts that the property-e is preserved under the smash product.


\begin{theorem}    \label{smash}
Let $(V_{\omega} , \Gamma_{\omega}, \rho_{\omega})_{ \omega \in \Omega }$ be a finite family of  
systems with property-e, each of them having at least one fixed point,
where $\Gamma_{\omega}$ is a given group or semigroup, and $V_{\omega}$ is a closed manifold.
Then the system $(  \wedge_{\omega \in \Omega} V_{\omega} ,     \wedge_{\omega \in \Omega} \Gamma_{\omega} , 
\wedge_{\omega \in \Omega} \rho_{\omega} )$
is expansive as well provided  the base points are taken as invariant ones for the representation
$(\Gamma_{\omega}, \rho_{\omega})$,   for every $\omega$ in $\Omega$. 

\end{theorem}

\begin{proof} For simplicity  consider the case when $\Omega$ has two elements, and
 $\Gamma_{\omega}$ coincides with $\N$ for
both $\omega$'s. So assume that $(V , \mathbb{N}, f)$ and $(W , \mathbb{N},  h)$ correspond to expansive actions of $\N$ on
$V$ and $W$ repectively, with $v_{0}$ and $w_{0}$ being fixed points for $f$ and $h$,
respectively. Then the  system $( V \wedge W, \mathbb{N}^{2}, f \wedge h)$ corresponds to an action of
$\mathbb{N}^{2}$  on $V \wedge W$.

Let $v_{0}$ and $w_{0}$ denote the base points of $V$ and $W,$ to construct a family
$\{ g(t) \}_{ t  \in\  ]0 , 1] }$ of Riemannian metrics on $V \times W$ as follows. Let
$\kappa : V \to [0,1]$ and $\rho : W \to [0,1]$ be functions  different from zero outside the
base points, smooth enough, but such that
$\lim_{v \to v_{0}} \kappa(v) = 0$ and $\lim_{w \to w_{0}} \rho(w) = 0$. Define
for each $t$ in $]0,1]$ the  Riemannian metric $g(t)$ on $V \times W$   by

$$ g(t) := (\  ( t + (1-t) \rho )  \ g_{V}\  )   \oplus (\  ( t  + (1 - t ) \kappa      )  \ g_{W}\    )  , $$
where $g_{V}$ and $g_{W}$ are   metrics on $V$ and $W,$ both of finite
diameter and of a suitable regularity.

Denote by $d_{g(t)}$ the distance on $V \times W$ induced by $g(t),$ and consider
the family of metric spaces $\{\ ( V \times W , d_{g(t)} ) \}_{ t \in\ ]0,1] }$.
As $t$ goes to zero the couple $( V \times W , d_{g(t)} ) $ ceases to be a metric space because
all the elements in $V \lor W$ (recall that base points are understood) are at zero distance.

After those remarks it is interesting to note:
\begin{lemma}
One has the convergence
$$( V \times W , d_{g(t)} )  \rightsquigarrow ( V \wedge W , d_{g(0)}) $$
in the Gromov-Hausdorff  sense as $t$ goes to zero (see \cite{verde}).
\end{lemma}

In the Gromov-Hausdorff metric space  identify $(V \times W , d_{g(t)})$ with
$(V \times W)_{t}$ for every $t$ in $[0,1]$, to denote by
$$\{ (\  (V \times W)_{t} , \mathbb{N}^{2} , (f \times h)_{t}\  )\ \}_{ t \in [0,1] }  $$
the collection of systems  obtained, where
$$(\   (V \times W)_{0} , \mathbb{N}^{2} ,  (f \times h)_{0}\  ) = 
(\  (V \wedge W , d_{g(0)} ) ,  \mathbb{N}^{2} ,   f \wedge h \  ).$$

Since by assumption both $(V ,\mathbb{N}, f)$ and $(W ,\mathbb{N}, h)$ are expansive systems, we conclude
that $(\  (V \times W)_{t} , \mathbb{N}^{2}, (f \times h)_{t}\ )$ is also expansive 
for every $t $ different from zero; indeed, the property of being expansive is a conjugacy invariant that
does not depend on the metric chosen (see Section \ref{expansivity}).

To conclude the proof  we add further conditions to  the functions $\kappa$ and $\rho$  to ensure 
the expansive property on $( V \wedge W  , \mathbb{N}^{2} , f \wedge h)$ 
\textit{thanks to the metric} $d_{g(0)}$.

Let $c_{V}$ and $c_{W}$ be expansivity constants for $( V , d_{V} , \mathbb{N}, f )$ and
$(W , d_{W}, \mathbb{N} , h),$ where $d_{V}$ and $d_{W}$ are the distance functions induced by
the Riemannian metrics $g_{V}$ and $g_{W},$ respectively. If $d_{V}(v , v_{0})$ is larger
than $c_{V}$ we require that $\kappa(v) = 1$, and if $d_{W}(w , w_{0})$ is bigger than
$c_{W}$ we demand that $\rho(w) = 1$.

Denote by $[ v , w]$  the point in $V \wedge W$ that is the image of   $(v , w)$
under the map from $V \times W$ to $V \wedge W$.
Observe that
$[ v , w_{0}] = [ v_{0} , w] = [ v_{0} , w_{0}]$ for every $(v , w)$ in $V \times W$, where
$[ v_{0} , w_{0} ]$ is a fixed point for
$$f^{n_{1}} \wedge h^{n_{2}} = ( f^{n_{1}} \wedge 1_{W}) \cdot ( 1_{V} \wedge h^{n_{2}}) =
( 1_{V} \wedge h^{n_{2}}  ) \cdot ( f^{n_{1}} \wedge 1_{W}  )$$
whenever $( n_{1} , n_{2} )$ is in $\N^{2}$.

Choose different  points $[ v , w]$ and $[ v' , w']$ in $V \wedge W,$ and exhaust all the possibilities
to infer the expansiveness of $( (V \wedge W  , d_{g(0)}) , \mathbb{N}^{2} , f \wedge h  )$ with  
e-constant $\min \{  c_{V} , c_{W}   \}$.

The extension to the general case is direct.
\end{proof}


 Consider the case when $\Omega$ has $d$ elements, and
 for every $\omega$ in $\Omega$ choose $(V_{\omega} , \Gamma_{\omega} , \rho_{\omega} )$ as being
 conjugated to $( S^{1} , \N , f )$
 with the $\N$-action on $S^{1}$  given by $f (\theta) = 2 \theta.$ Take $\theta = 0$
 as the base point in $S^{1}$ to construct the bouquet of $d$ circles $\lor_{d}\ S^{1}$, and note that
 $\wedge_{d}\ S^{1} = S^{d}$ is endowed with an action of $\N^{d}$ induced by $\wedge^{d}_{i=1} f^{n(i)}$.

Theorems \ref{generator-generates} and \ref{smash} together with Lemma \ref{conjugacy} give, thanks to 
(the proof of) the generalized Poincar\'e conjecture\footnote{Finished in the work of  W. Thurston, 
R. Hamilton and G. Perelman in dimension 3, M. Freedman in dimension 4, and S. Smale in higher dimensions.}: 

 \begin{corollary} \label{Sd}
 Let $V$ be an homotopy $S^{d}$. Then  there exists an expansive action 
 of $\N^{d}$ on $V$. If $\alpha$
 is a generator for $( V , \N^{d}, \wedge^{d} f )$ and $\{ F(n) \}_{n \in \N }$ is an increasing  sequence
 exhausting $\N^{d}$, then when  $n$ goes to infinity  the nerve of  
 $\alpha_{F(n)}$ is homeomorphic to
 $S^{d}$.
 \end{corollary}


\bigskip

Universidad Andres Bello

 \smallskip

 Departamento de Matem\'aticas

\smallskip

 Facultad de Ciencias Exactas

 \smallskip

 Rep\'ublica 220, Piso 2

 \smallskip

 Santiago, Chile

 \bigskip

\smallskip


\end{document}